\documentclass{amsart}
\usepackage[utf8]{inputenc}

\usepackage{graphics}
\usepackage{thmtools}
\usepackage{wasysym}
\usepackage{mathptmx}
\usepackage[T1]{fontenc}    

\usepackage{amsthm}
\usepackage{amsbsy,amsmath,amssymb,amscd,amsfonts}
\usepackage[pagebackref=false]{hyperref}
\usepackage[nameinlink,capitalize,noabbrev]{cleveref}

\usepackage{graphicx,float,latexsym,color}
\usepackage[font={scriptsize,it}]{caption}
\usepackage{subcaption}
\usepackage[export]{adjustbox}

\usepackage{makecell}

\usepackage[dvipsnames]{xcolor}

\newtheorem{theorem}{Theorem}
\newtheorem*{theorem*}{Theorem}
\newtheorem{observation}{Observation}
\newtheorem{proposition}{Proposition}

\newtheorem{corollary}{Corollary}
\newtheorem{lemma}{Lemma}
\theoremstyle{remark}

\theoremstyle{definition}

\hypersetup{
    pdftoolbar=true,        
    pdfmenubar=true,        
    pdffitwindow=false,     
    pdfstartview={FitH},    
    colorlinks=true,       
    linkcolor=OliveGreen,          
    citecolor=blue,        
    filecolor=black,      
    urlcolor=red           
}

\usepackage{lineno}

\usepackage{comment}
\usepackage{microtype}
\usepackage{footnote}
\newcommand{\A}{\mathcal{A}}
\renewcommand{\H}{\mathcal{H}}

\newcommand{\T}{\mathcal{T}}
\newcommand{\C}{\mathcal{C}}

\renewcommand{\T}{\mathcal{T}}

\title[A Web of Confocal Parabolas in a Grid of Hexagons]{A Web of Confocal Parabolas\\in a Grid of Hexagons\vspace{-0.75em}}
\author{Peter Moses}
\thanks{P. Moses, Moparmatic Inc., Worcestershire, England, \texttt{moparmatic@gmail.com}}
\author{Dan Reznik$^*$} 
\thanks{D. Reznik, Data Science Consulting Ltd., Rio de Janeiro, Brazil.
\texttt{dreznik@gmail.com}}
\date{April, 2022}

\begin{document}

\maketitle

\vspace{-1cm}
\begin{abstract}
If one erects regular hexagons upon the sides of a triangle $T$, several surprising properties emerge, including: (i) the triangles which flank said hexagons have an isodynamic point common with $T$, (ii) the construction can be extended iteratively, forming an infinite grid of regular hexagons and flank triangles, (iii) a web of confocal parabolas with only three distinct foci interweaves the vertices of hexagons in the grid. Finally, (iv) said foci are the vertices of an equilateral triangle.
\vskip .3cm
\noindent\textbf{Keywords} hexagon, flank, map, isodynamic, parabola, confocal.
\vskip .3cm
\noindent \textbf{MSC} {51M04
\and 51N20 \and 51N35\and 68T20}
\end{abstract}

\section{Introduction}
Napoleon's theorem, studied in \cite{martini1996}, is illustrated in \cref{fig:napoleon}(left): the centroids of the 3 equilaterals erected  upon the sides of a reference triangle $T=ABC$ are vertices of an equilateral known as the ``outer'' Napoleon triangle \cite{mw}. A related construction due to Lamoen appears in \cref{fig:napoleon}(right), whereby squares are erected upon the sides of $\T$ and triangles which ``flank'' said squares are defined \cite{lamoen2004-flank}. 

\begin{figure}
    \centering
    \includegraphics[trim=50 200 100 20,clip,width=\textwidth,frame]{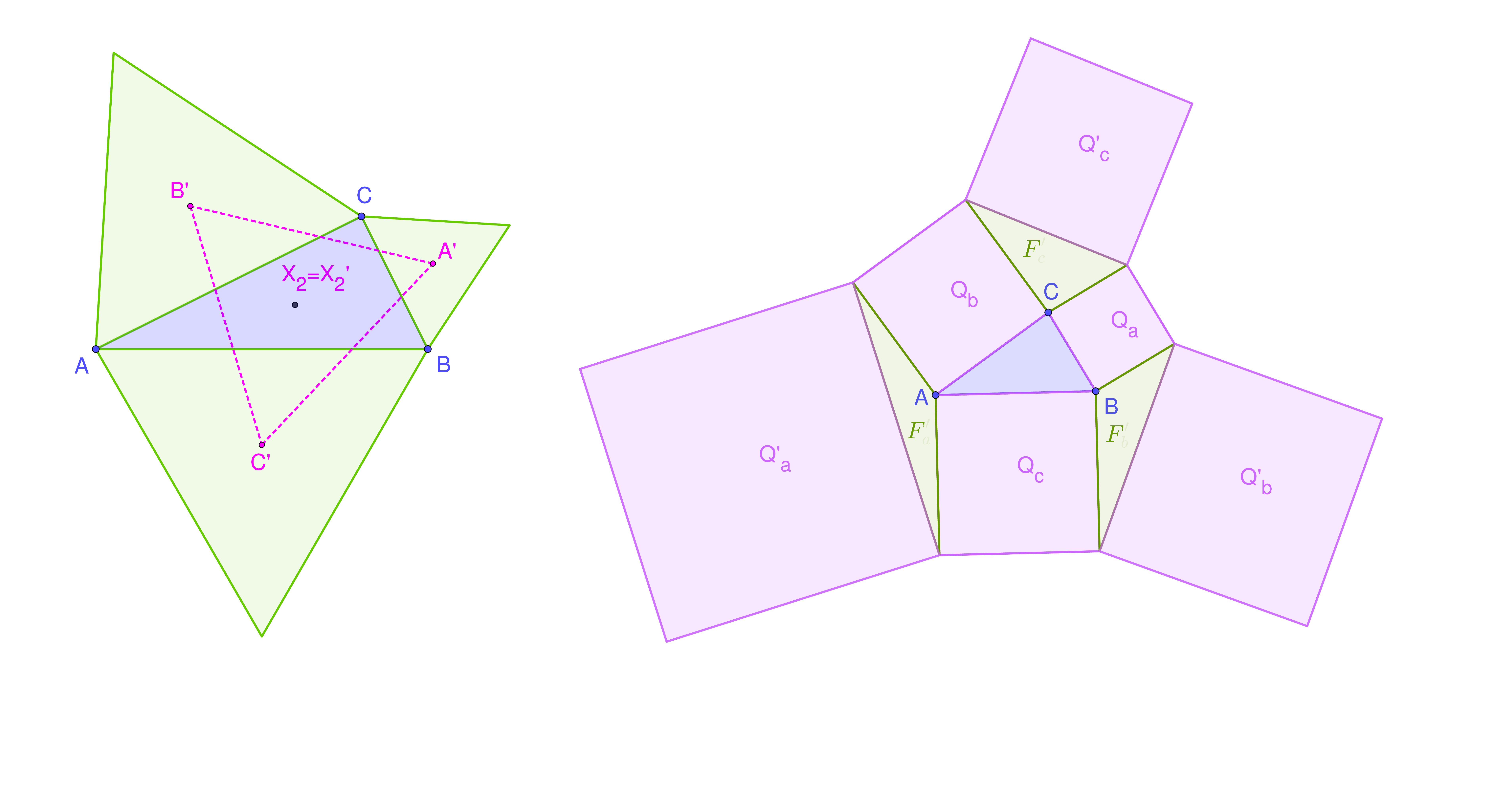}
    \caption{\textbf{Left:} Napoleon's theorem states that a triangle $\T'=A'B'C'$ (dashed magenta) whose vertices are the centroids of 3 equilaterals (green) erected upon the sides of a triangle $\T=ABC$ (blue) is itself equilateral. Furthermore, its centroid coincides with the barycenter $X_2$ of $\T$ \cite[Napoleon's Thm]{mw}. \textbf{Right:} Lamoen \cite{lamoen2004-flank} has studied properties of ``flank'' triangles $F$ defined between neighboring squares $Q_a,Q_b,Q_c$ erected upon the sides of a $\T=ABC$, as well as those of additional squares $Q_a',Q_b',Q_c'$ erected upon the flanks.}
    \label{fig:napoleon}
\end{figure}

We borrow from those ideas and study a construction whereby regular {\em hexagons} $\H_a,\H_b,\H_c$ are erected upon the sides of $\T$, with 3 ``flank'' triangles $F_a,F_b,F_b$ defined between them, see  \cref{fig:basic}. As it turns out, many surprising properties emerge, listed below.

\begin{figure}
    \centering
    \includegraphics[trim=300 50 600 50,clip,width=.8\textwidth]{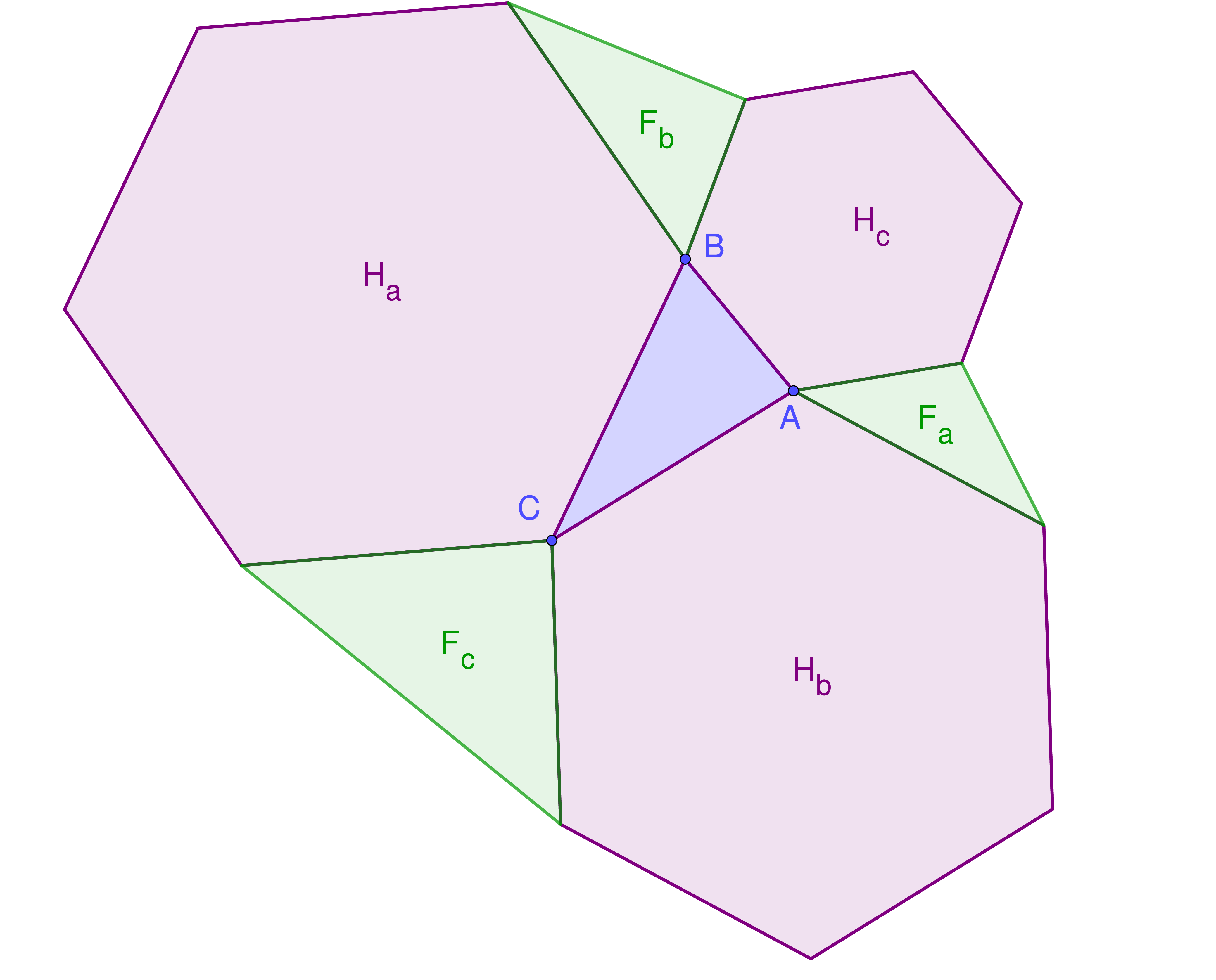}
    \caption{Given a reference triangle $\T=ABC$, erect three regular hexagons $\H_a,\H_b,\H_c$ on its sides; these define three flank triangles with a vertex at $A$, $B$, or $C$ (respectively), and two others which are proximal vertices of erected hexagons sharing $A$, $B$, or $C$, respectively.}
    \label{fig:basic}
\end{figure}

\subsection*{Summary of the Results}

\begin{itemize}
    \item The second isodynamic points \cite{mw} of $\T$ and the flanks are common;
    \item A contiguous, infinite grid of regular hexagons can be constructed, see \cref{fig:grid}; all flanks in the grid have a common second isodynamic point and conserve a special quantity;
    \item Sequences of hexagon vertices are crisscrossed by a web of confocal parabolas with only 3 distinct foci;
    \item Their foci and directrix intersections are vertices of 2 new equilateral triangles;
    \item Phenomena are described when the reference triangle in \cref{fig:basic} are triangles in two special Poncelet triangle families (for a survey on Poncelet's porism, see \cite{dragovic2014}).
\end{itemize}

\subsection*{Related Work}

In \cite{cerin2002-flanks,lamoen2004-flank} properties of ``flank'' triangles located between squares and/or rectangles erected on a triangle's sides are studied. Works \cite{dosa2007-ext,hoehn2001-ext} study triangles centers (taken as triples or not) analogues of the intouch and/or extouch triangles erected upon each side of a reference triangle. In \cite{fukuta1996,fukuta1997,stachel2002,cerin1998} a construction related to Napoleon's theorem is described which associates to a generic triangle a regular hexagon. A study of properties and invariants of regular polygons erected upon the homothetic Poncelet family appears in \cite{fedor2022}.

\subsection*{Article Structure}

In \cref{sec:trio} we describe properties of the trio of flank triangles. In \cref{sec:satellite} we fix a central regular hexagon and consider properties of 6 ``satellite'' triangles built around it, and this implies a contiguous grid can be iteratively built. . In \cref{sec:parabolas} we show said grid is interwoven by three groups of confocal parabolas. In \cref{sec:poncelet} we show two examples of grids controlled by Poncelet poristic triangles.  \cref{sec:videos} provides a table of videos illustrating some results as well as a list of open questions. In \cref{app:eqns} we provide computer-usable expressions for some objects described in \cref{sec:parabolas}.

\section{Hexagonal flanks}
\label{sec:trio}
Adhering to Kimberling's $X_k$ notation for triangle centers \cite{etc}, recall a triangle's {\em isodynamic points} $X_{15}$ and $X_{16}$ are the limiting points of a pencil\footnote{This pencil is known as the {\em Schoutte} pencil \cite{johnson17-schoutte}.} of circles containing a triangle's circumcircle and the Brocard circle \cite[Isodynamic Points]{mw}. For foundations and properties of Brocard geometry and/or the isodynamic points see \cite{casey1888,johnson1960}.

\begin{figure}
    \centering
    \includegraphics[trim=300 0 400 0,clip,width=\textwidth]{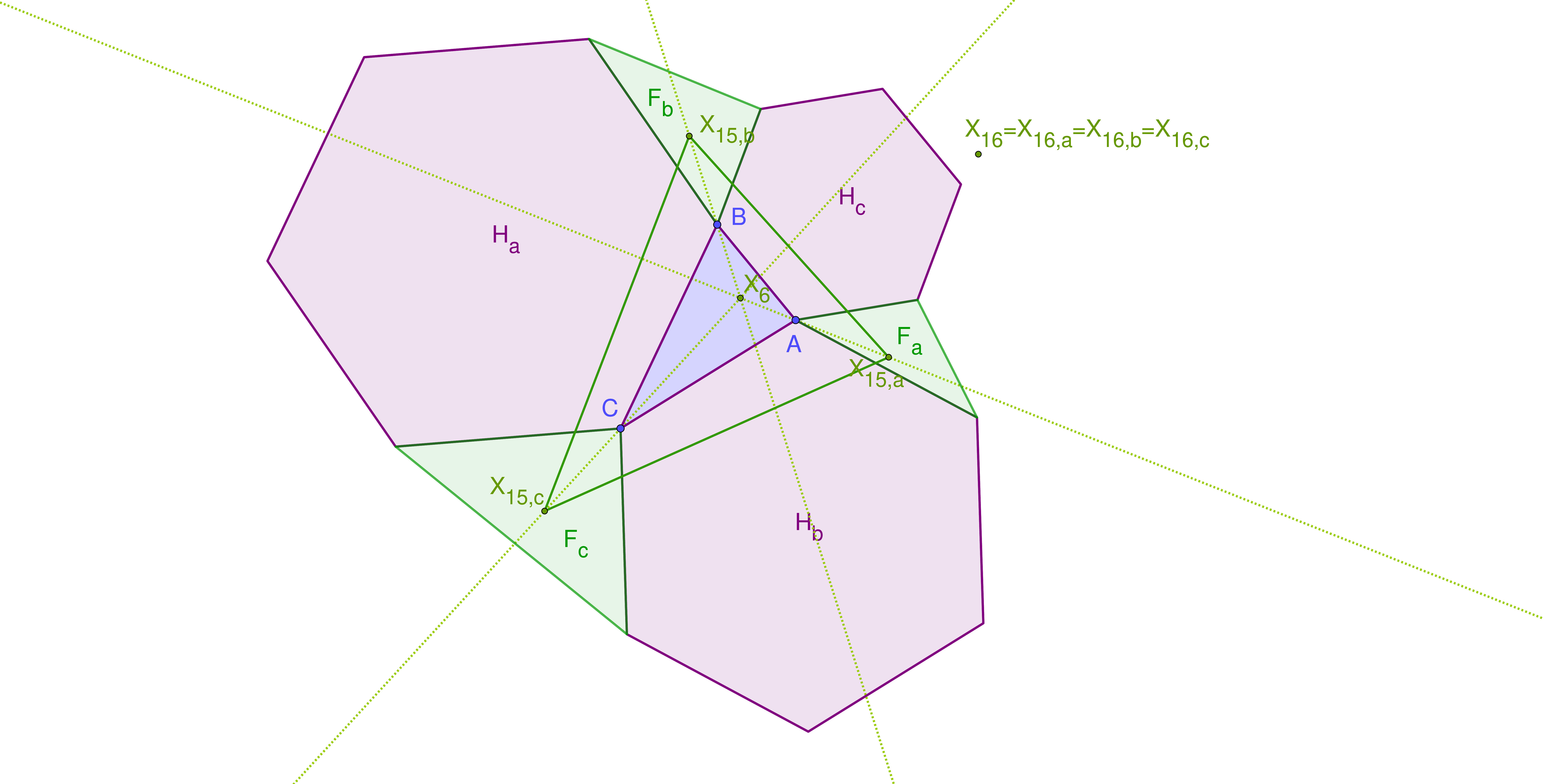}
    \caption{(i) the triangle formed by the 1st isodynamic points $X_{15,a}$, $X_{15,b}$, $X_{15,c}$ of the 3 flank triangles is perspective with $\T=ABC$ at the latter's $X_6$; (ii) the 2nd isodynamic point of $\T$ and the 3 flank triangles are common (upper right).}
    \label{fig:trio}
\end{figure}

Referring to \cref{fig:trio}:

\begin{proposition}
The triangle with vertices on the $X_{15}$-of-flanks is perspective with $\T$ at the latter's symmedian point $X_{6}$.
\label{prop:x6}
\end{proposition}

\begin{proof}
A straightforward derivation for the barycentrics of $X_{15}$ of the $A$-flank yields $(a^2-3 b^2-3 c^2,b^2,c^2)$, and those of the $B$- and $C$-flanks can be obtained cyclically. Recall from the barycentrics of $X_6$ are $(a^2,b^2,c^2)$ (see \cite{etc}), so the result follows.
\end{proof}

\begin{proposition}
The 2nd isodynamic points $X_{16}$ of the three flank triangles coincide with that of $\T$.
\end{proposition}

The following proof sketch was kindly contributed by D. Grinberg \cite{darij2021-private}:

\begin{proof}
Let two triangles $\T=ABC$ and $\T'=A'B'C'$ be called ``30-twins around B'' iff $B = B'$ and $\measuredangle BAA' = \measuredangle AA'B = 30^o$ and $\measuredangle BC'C = \measuredangle CC'B = 30^o$ (angles are in degrees and are directed angles mod $180^o$). In other words, this is the relation between $\T$ its flank triangle at B. It is a symmetric relation, up to relabeling $ABC$ as $CBA$ to get the correct orientation.

Let $ABC$ and $A'BC'$ be two triangles that are 30-twins around $B$. Fix a nontrivial circle $K$ around $B$, and let $X$, $Z$, $X'$ and $Z'$ be, respectively, the images of $A$, $C$, $A'$ and $C'$ under the inversion with respect to $K$. It can be shown that it follows that triangles $XBZ$ and $X'BZ'$ are again 30-twins around B.

The same inversion takes  $X_{16}$ of $ABC$ to the third vertex of an equilateral triangle erected inwardly on the side $XZ$ of triangle $XBZ$. The proof is an angle chase, using the the fact that $X_{16}$ is the unique point of a triangle which sees the three vertices at $120^o$ angles \cite[Isodynamic point]{mw}. The same applies for $X_{15}$ and an equilateral triangle erected outwardly. Likewise for triangle $A'BC'$ and $X'Z'$.

Thus, it remains to be shown that the third vertex of an equilateral triangle erected inwardly on the side $XZ$ of triangle $XBZ$ is identical to the third vertex of an equilateral triangle erected inwardly on side $X'Z'$ of triangle $X'BZ'$. This is easy to prove using complex numbers: $w^2 + w + 1 = 0$ with $w = e^{2pi i/3}$.
\end{proof}

\subsection{Zero-Area Flanks}

Consider a family of triangles $ABC$ where $A,B$ are fixed and $C$ is free. Let $F_c$ denote the flank triangle between the regular hexagons erected upon $AC$ and $CB$. As shown in \cref{fig:sliver}:

\begin{observation} 
$F_c$ will be zero-area if $C$ subtends a $120^o$ angle, i.e., it lies on a circular arc centered on the centroid $O$ of an equilateral erected upon $AB$ and with radius $|OA|$.
\end{observation}

\begin{figure}
    \centering
    \includegraphics[width=\textwidth] {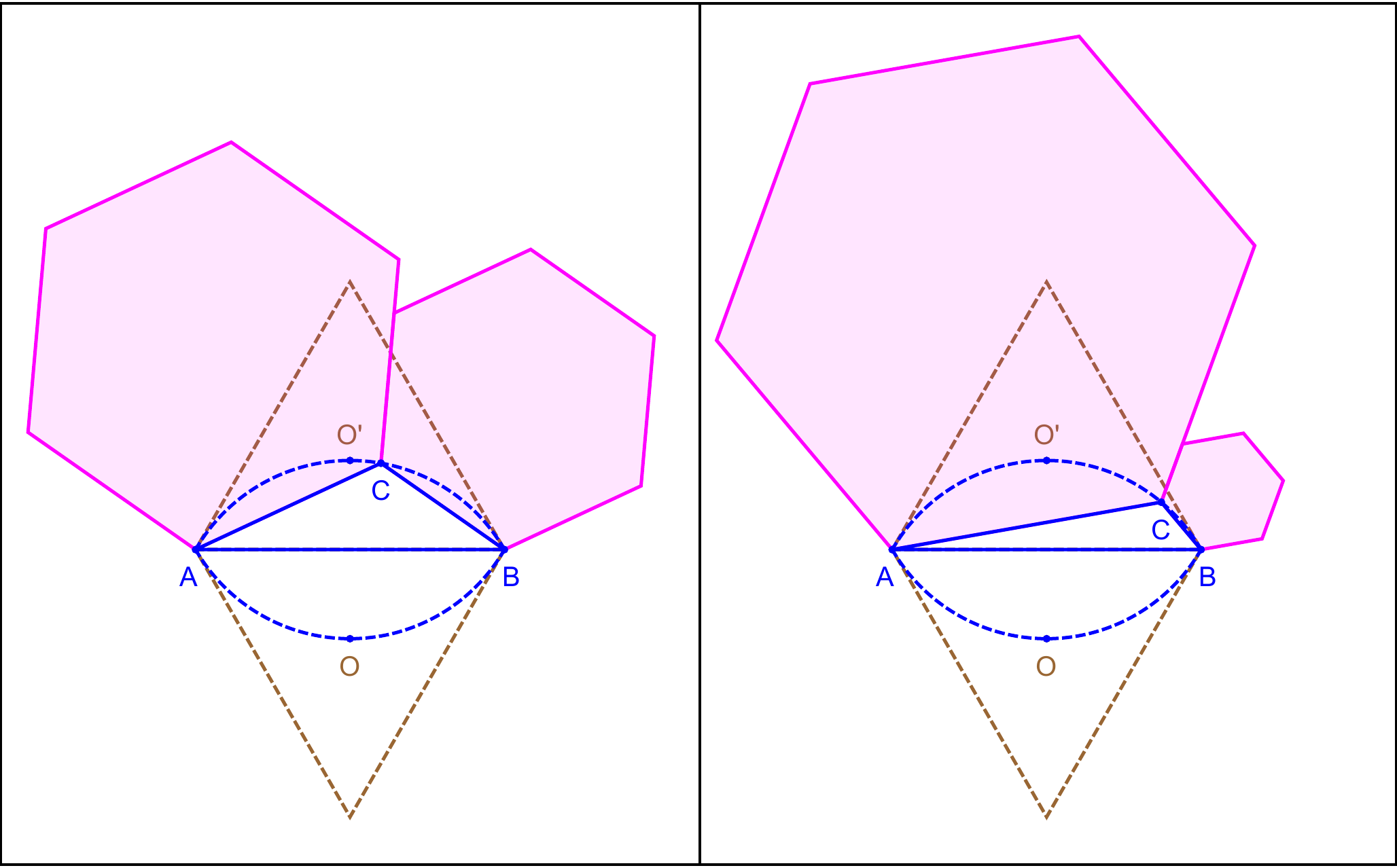}
    \caption{The locus of vertex $C$ such that it subtends a $120^o$ angle, i.e., the $C$-flank is degenerate, is a circular arc (dashed blue) such centered on $O$, the centroid of an equilateral (dashed brown) erected upon $AB$, and with radius $|OA|$ (a mirror arc corresponding to the top equilateral centered on $O'$ is also shown). The left (resp. right) picture shows $C$ in two distinct positions.}
    \label{fig:sliver}
\end{figure}

Consider the construction for the two flank triangles $F_1$ and $F_2$ shown in \cref{fig:self-inter}, namely, departing from $\T=ABC$, erect regular hexagons $\H_1$ and $\H_2$ on sides $AC$ and $BA$, respectively. Consider the a first flank triangle $F_1$ between $\H_1$ and $\H_2$. As shown in the figure, erect a third regular hexagon $\H_3$ on the unused side of $F_1$, and let a second flank triangle $F_2$ sit between $\H_2$ and $\H_3$. While holding $B$ and $C$ fixed, there are positions for $A$ such that $F_2$ has positive, zero, or negative signed area. Referring to \cref{fig:zero-f2}:

\begin{proposition}
$F_2$ will have positive (resp. negative) area if $A$ is exterior (resp. interior) to the circumcircle $\C$ of an equilateral triangle whose base is $B$ and the midpoint of $BC$. In particular, the vertices of $F_2$ become collinear if $A$ is on $\C$.
\end{proposition}

\begin{figure}
\begin{subfigure}[b]{0.32\textwidth}
\includegraphics[trim=10 0 80 20,clip,width=1\textwidth]{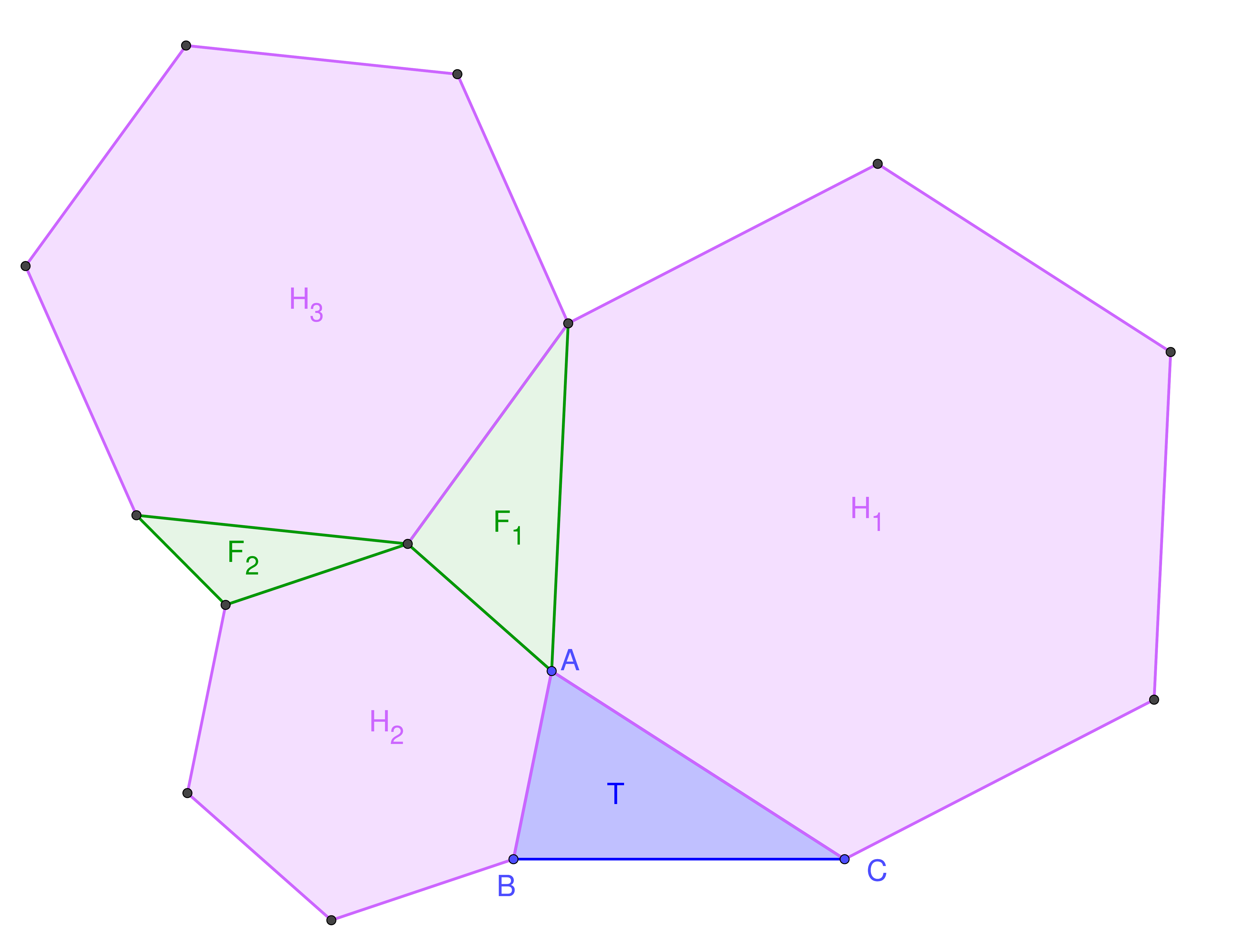}
\end{subfigure}
\begin{subfigure}[b]{0.32\textwidth}
\includegraphics[trim=80 0 70 20,clip,width=1\textwidth]{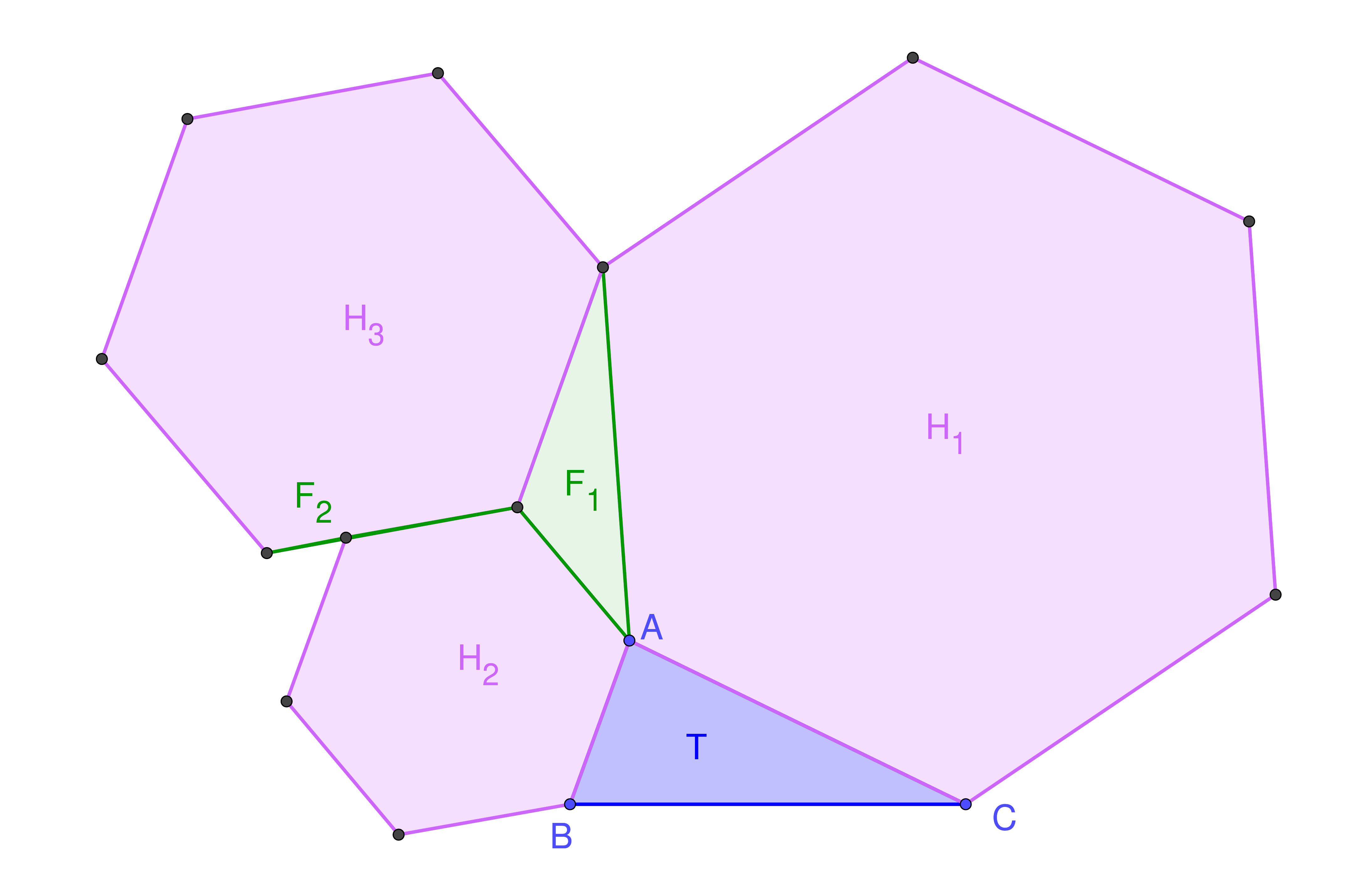}
\end{subfigure}
\begin{subfigure}[b]{0.32\textwidth}
\includegraphics[trim=30 0 50 20,clip,width=1\textwidth]{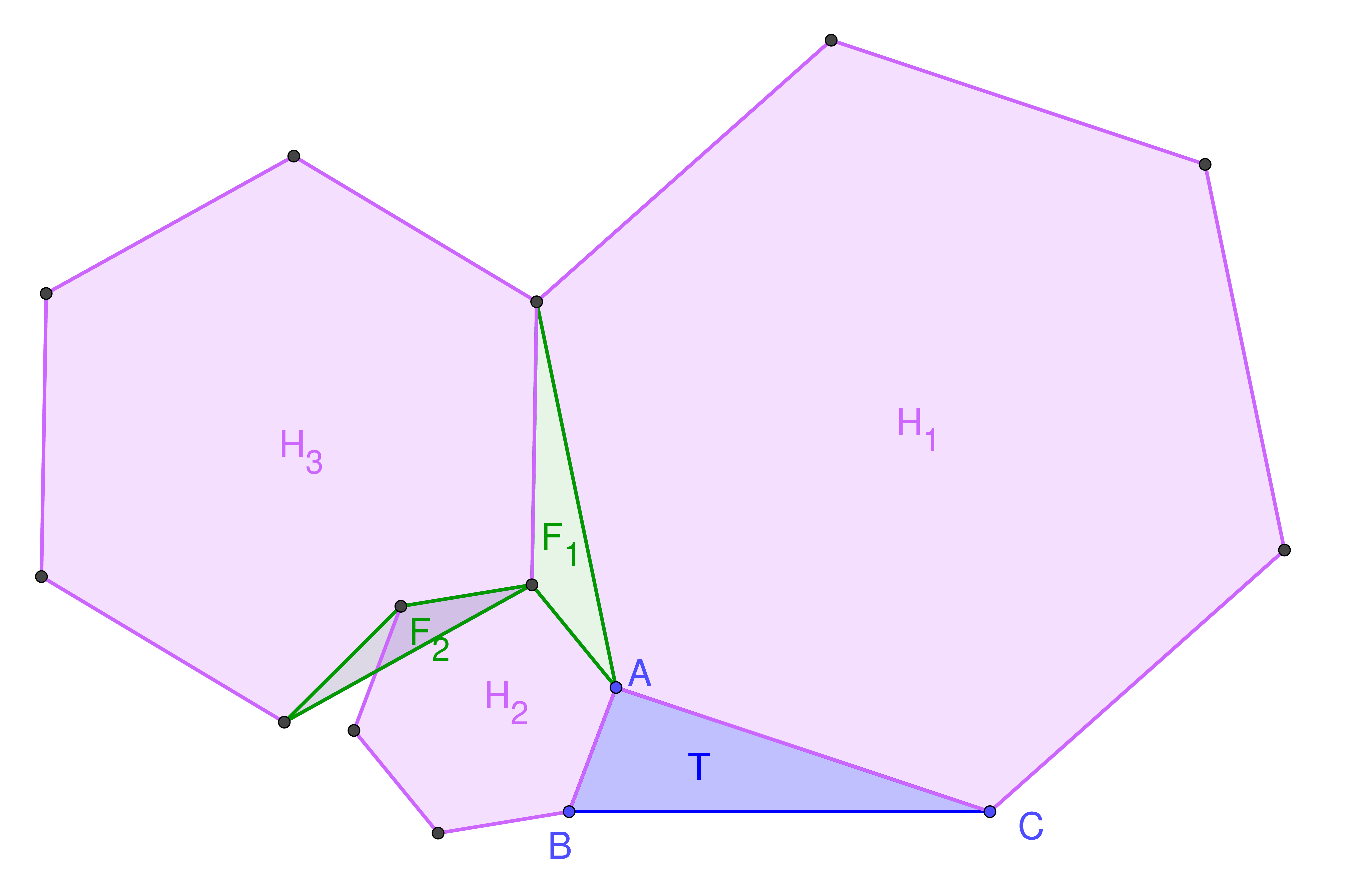}
\end{subfigure}
\caption{For certain positions of $A$ (while keeping $B$ and $C$ stationary), flank triangle $F_2$ will be non-eversed (left), degenerate (middle), or eversed (right).}
\label{fig:self-inter}
\end{figure}

\begin{figure}
    \centering
    \includegraphics[width=\textwidth]{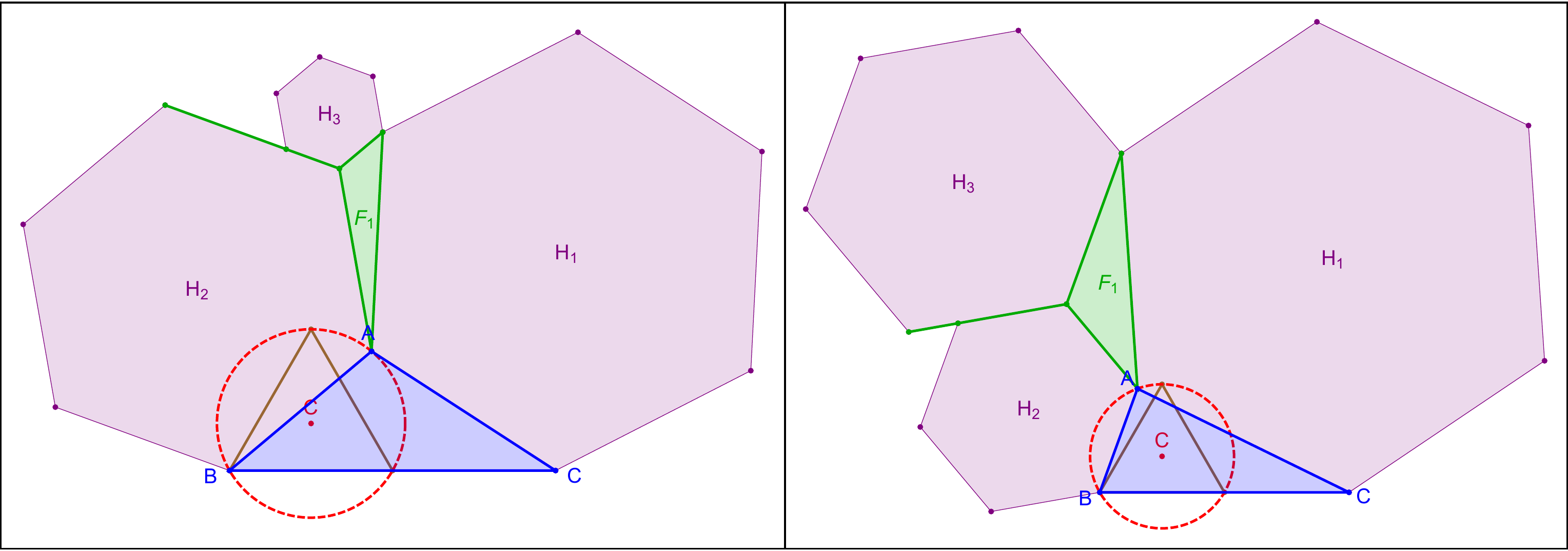}
    \caption{With $B$ and $C$ fixed, the locus of $A$ such that the area of $F_2$ oin \cref{fig:self-inter} is zero is the circumcircle (dashed red) of an equilateral (brown) with base on $B$ and the midpoint of $BC$. Two positions of $A$ are shown (left, right).}
    \label{fig:zero-f2}
\end{figure}

\section{Satellite triangles}
\label{sec:satellite}
Consider the construction shown in \cref{fig:closure}: given a triangle $\T=ABC$, erect regular regular hexagons $\H_0$ and $\H_1$ on sides $BC$ and $AB$. These define a first flank triangle $F_1$. Continue adding hexagons $\H_2,\ldots,H_5$, defining with $\H_0$ new flank triangles $F_2,\ldots,F_5$. Let $D$ (resp. $E$) be a vertex of $F_5$ which is common to $\H_0$ and $\H_5$ (resp. located on $\H_5$ and adjacent to $D$).

Let $A'$ be the reflection of $A$ about $BC$. The following lemma was kindly contributed by A. Akopyan \cite{akopyan2021-private}:

\begin{lemma}
The reflection of the apexes of $F_1,\ldots,F_5$ all coincide with $A'$.
\label{lem:refl-apices}
\end{lemma}

\begin{figure}
    \centering
    \includegraphics[trim=200 50 250 0,clip,width=.7\textwidth,frame]{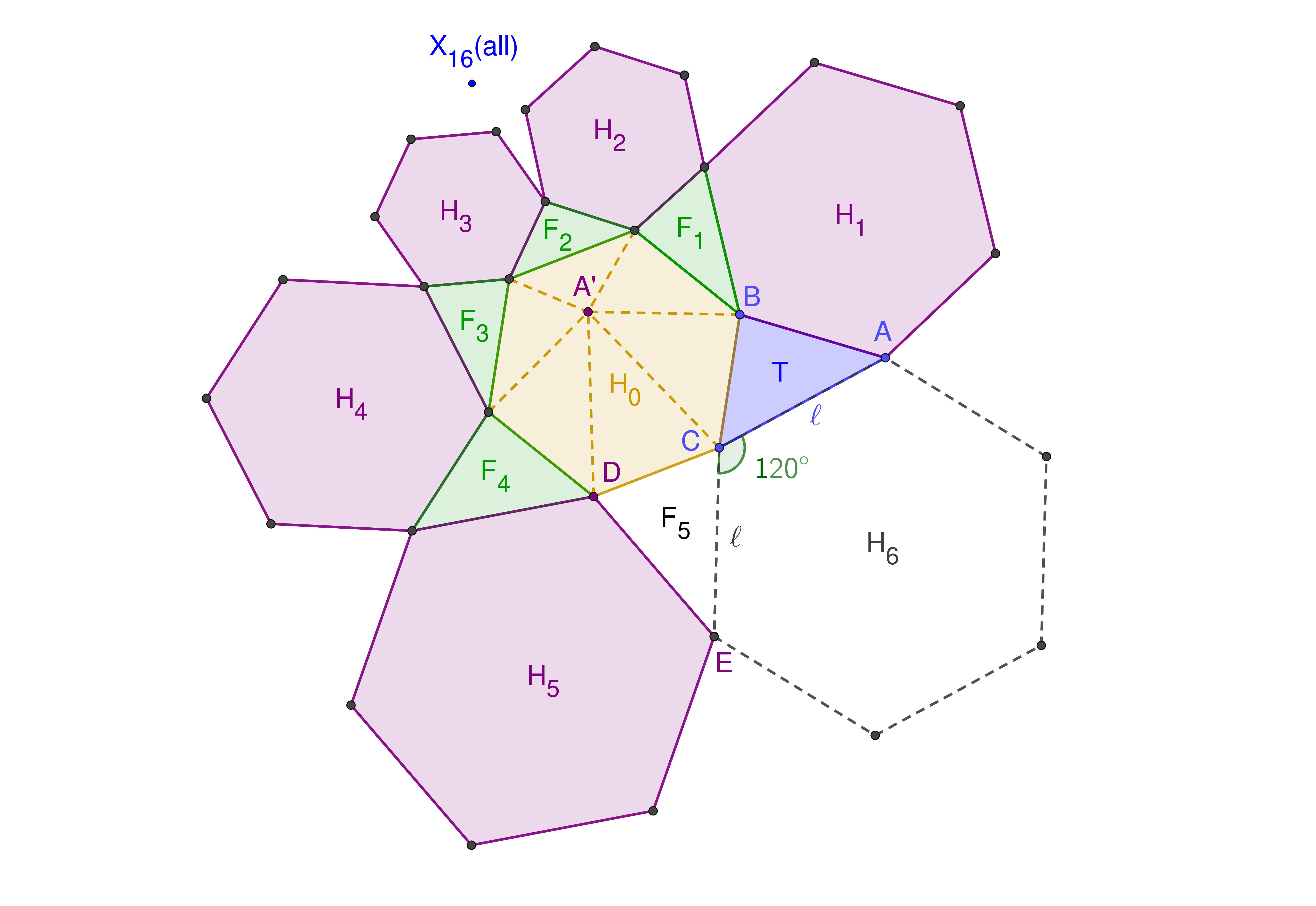}
    \caption{The grid ``closure'' property: departing from $\T=ABC$ add regular hexagons $\H_0,\ldots H_5$ and flanks $F_1,\ldots F_5$. The reflection of the apexes of $F_i$ about their based (sides of $\H_i$) is a common point $A'$. This implies the regular hexagon $\H_6$ erected on $AC$ fits perfectly between $\H_1$ and $\H_5$. The $X_{16}$ common to $\T$ and the five flanks is shown in the upper left of the picture.}
    \label{fig:closure}
\end{figure}

Still referring to \cref{fig:closure}, let $CE$ be the side of $F_5$ not on $\H_0$ nor on $\H_5$. We can use \cref{lem:refl-apices} to show that:

\begin{proposition}
$|CE|=|AC|$ and $\measuredangle ECA = (2\pi)/3$. 
\label{prop:snap}
\end{proposition}
This implies that a sixth, regular hexagon $\H_6$ can be erected on $AC$ and one of its vertices will snap perfectly against vertex $E$ of $\H_5$. Referring to \cref{fig:grid}:

\begin{corollary}
This construction can be extended ad infinitum, creating, modulo self-intersections, a locally-consistent contiguous grid.
\end{corollary}

Since in this infinite grid one can isolate a central triangle and the three flanks around it (see \cref{fig:basic}), you get the following propagation:

\begin{corollary}
The second isodynamic points $X_{16}$ of all flank triangles in any contiguous grid coincide in a single point.
\end{corollary}

\begin{figure}
    \centering
    \includegraphics[trim=150 150 150 50,clip,width=.8\textwidth]{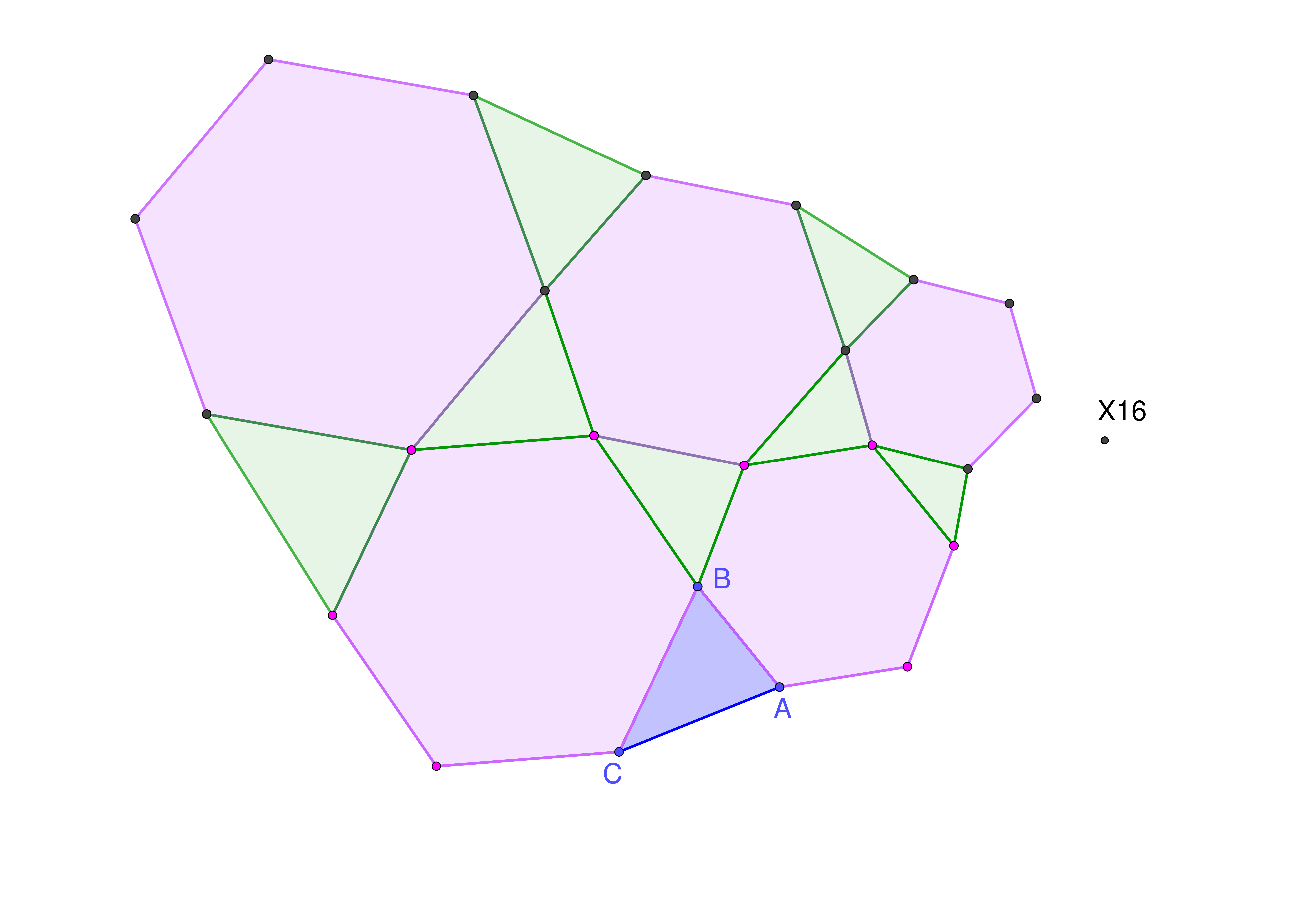}
    \caption{Departing from $\T=ABC$ a contiguous grid can be built by adding more triangles and regular hexagons as needed. The second isodynamic point $X_{16}$ of all interstitial triangles is common with that of $\T$.}
    \label{fig:grid}
\end{figure}

Consider the related construction in \cref{fig:satellite}: let $\H$ denote a fixed regular hexagon with vertices $Q_i$, $i=1,\ldots 6$. Given a point $P$, let $Q_1 Q_2 P$ be a first ``satellite'' triangles $F_1$. Create 5 new satellite triangles $F_i$ as follows: Let $F_{i}$ be the flank triangle obtained by erecting a regular hexagon on a side of $F_{i-1}$, $i=2,\ldots 6$.

\begin{figure}
    \centering
    \includegraphics[width=\textwidth]{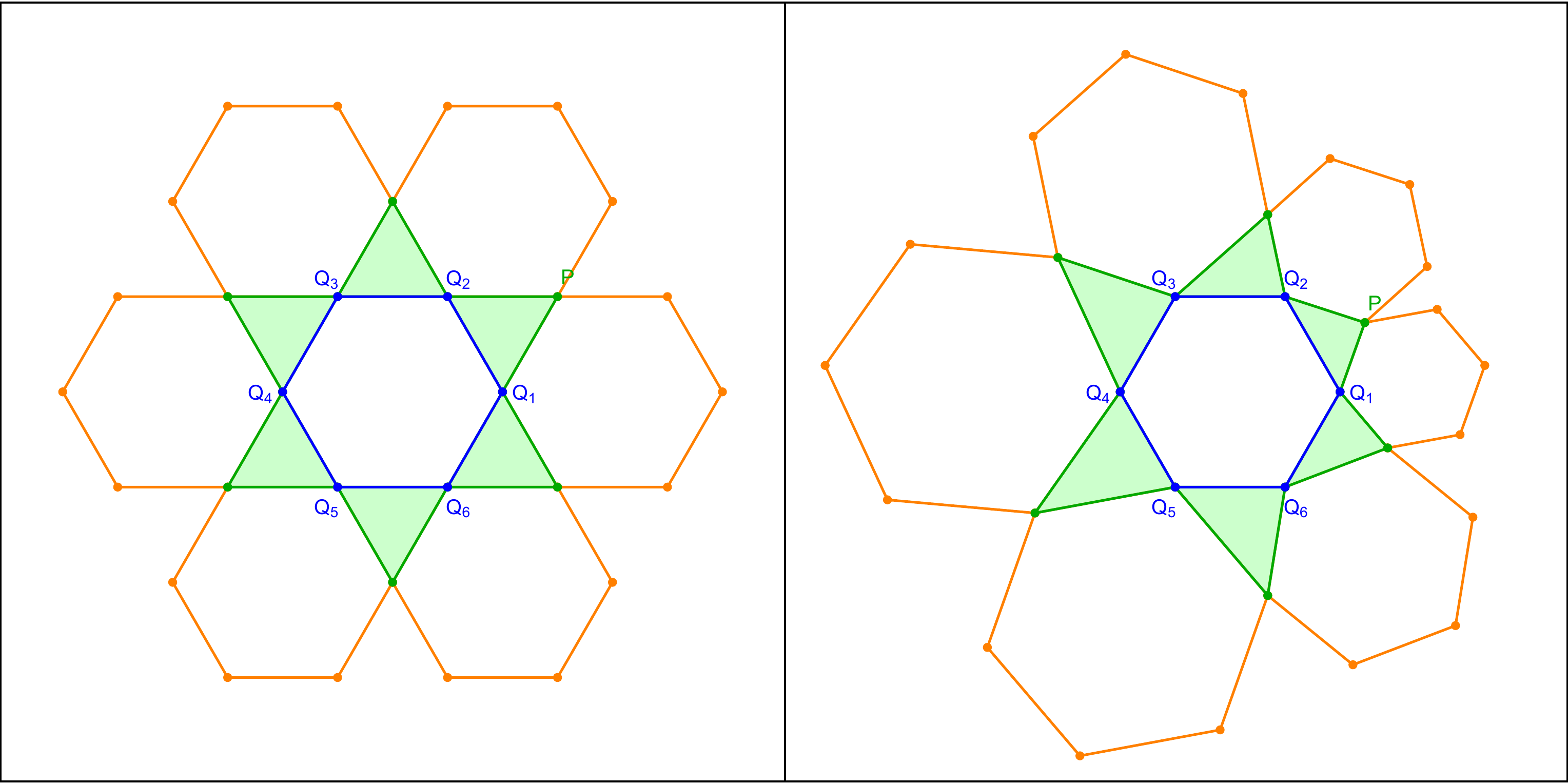}
    \caption{Given a fixed central hexagon $\H$ with vertices $Q_i,i=1,\ldots,6$, let 6 ``satellite'' flank triangles (green) be constructed around $\H$ departing from a first triangle $P Q_1 Q_2$. The sum of areas of said 6 satellites is independent of $P$ and equal to the area of $\H$. In the left (resp. right) $P$ is positioned so the construction is regular (resp. irregular).}
    \label{fig:satellite}
\end{figure}

From \cref{lem:refl-apices,prop:snap}, the reflected images of the flanks about their bases fill the interior of $\H$, therefore:

\begin{corollary}
The sum of the areas of the 6 satellite triangles is independent of $P$ and equal to the area of $\H$.
\end{corollary}

Referring to \cref{fig:conic}:

\begin{proposition}
The apexes of the 6 satellite triangles lie on a conic iff $P$ lies on either line $Q_3 Q_2$ or line $Q_6 Q_1$.
\end{proposition}

\begin{figure}
    \centering
    \includegraphics[width=\textwidth]{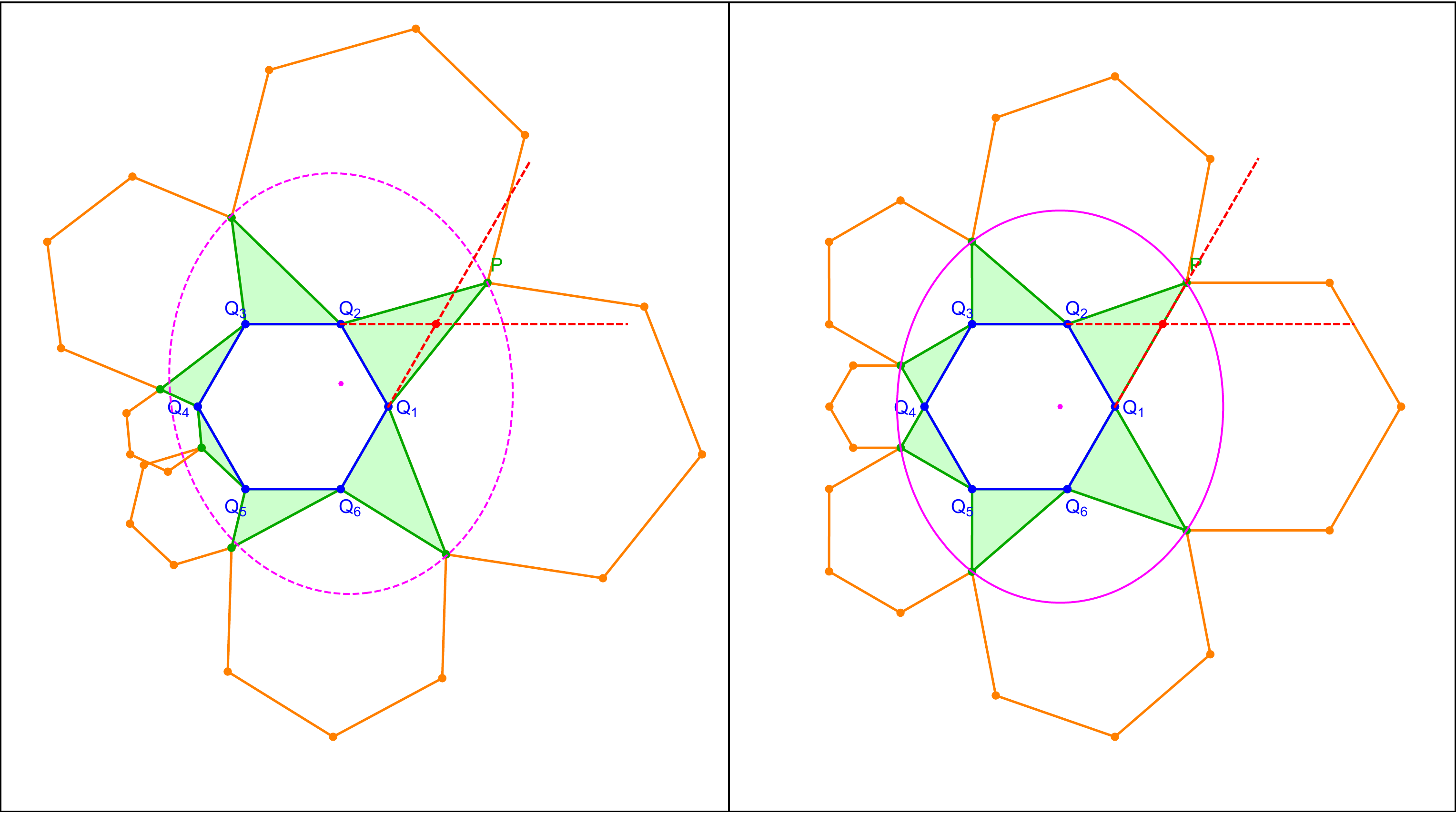}
    \caption{The apexes of the 6 satellite triangles (green) around a fixed hexagon (blue) lie on a conic iff $P$ is on either line $Q_3 Q_2$ or line $Q_6 Q_1$. When $P$ is at the intersection of said lines, the 6 apexes are concyclic, as in \cref{fig:satellite}(left).}
    \label{fig:conic}
\end{figure}

Referring to \cref{fig:iso-hex}, consider the 6 auxiliary regular hexagons erected successfully on the 6 satellite triangles parametrized by $P$. Let $O$ denote the intersection of $Q_3 Q_2$ and $Q_6 Q_1$. 

\begin{figure}
    \centering
    \includegraphics[width=.8\textwidth,frame]{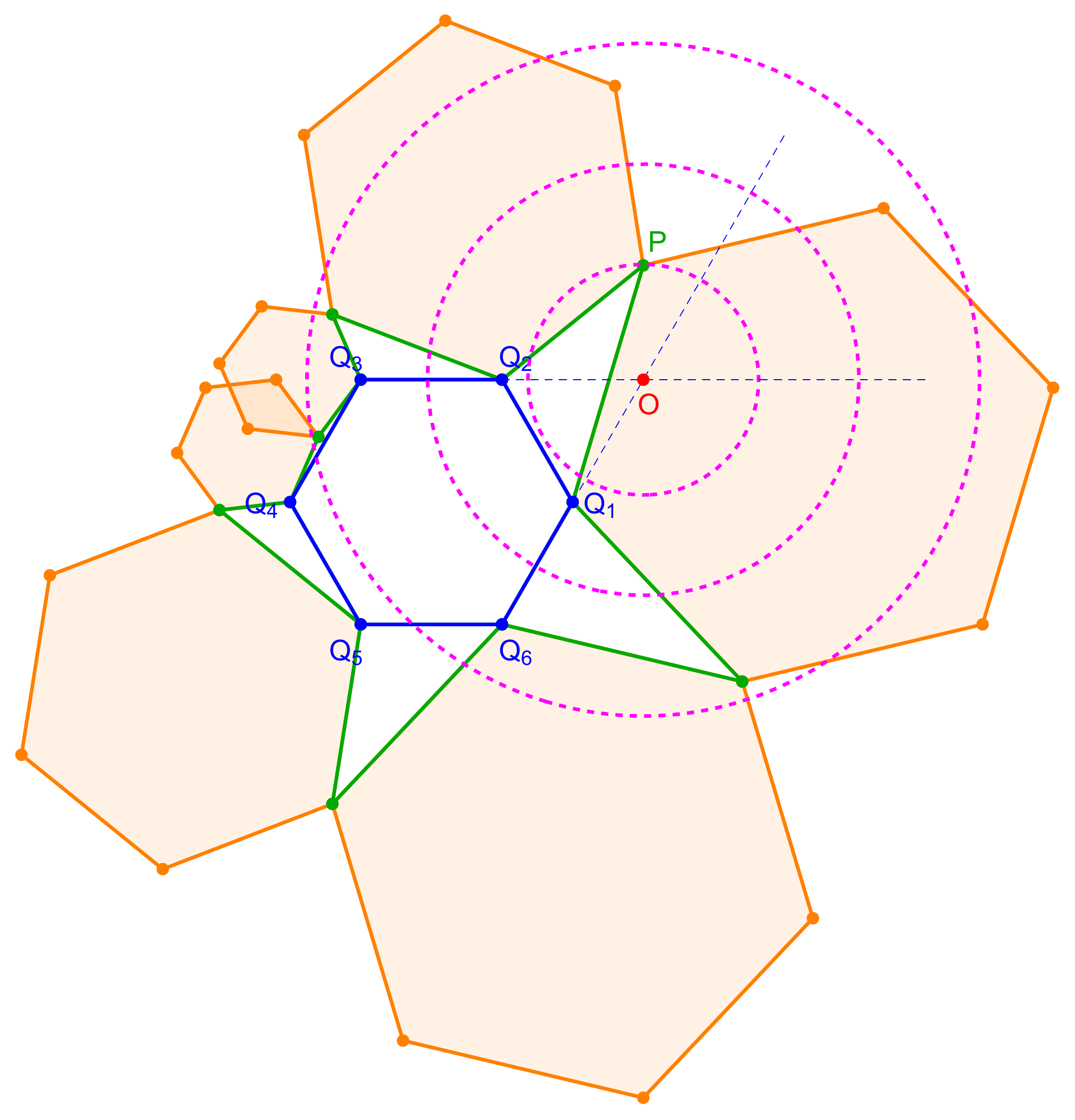}
    \caption{The iso-curves of $P$ such that the sum of areas of the satellite hexagons is constant are circles (dashed magenta) centered on $O$, at the intersection of the two sides of $\H$ adjacent to the base of $P Q_1 Q_2$.}
    \label{fig:iso-hex}
\end{figure}

\begin{proposition}
The iso-curves of $P$ such that the sum of the areas of the 6 satellite hexagons is constant are circles centered on $O$.
\end{proposition}

\subsection{Second-level satellites}

Referring to \cref{fig:second-level}, consider the 6 satellite flanks $F_i$ surrounding a central hexagon which are obtained as above by sequentially erecting 6 regular hexagons $\H_i$. Consider 6 ``2nd-level'' flanks $F_i'$ nestled between the $\H_i$. Let $\H_k'$ denote a hexagon whose vertices are the $X_k$ of the $F_i'$.

\begin{proposition}
For all $X_k$ on the Euler line, $\H_k'$ has invariant internal angles.
\end{proposition}

We thank A. Akopyan for the following argument \cite{akopyan2021-private}:

\begin{proof}
This follows from the fact that their area is the sum of squares of distances from the reflected point $A'$ to vertices $Q_i$ of the central hexagon.
\end{proof}

\begin{figure}
    \centering
    \includegraphics[trim=120 0 150 0,clip,width=.8\textwidth,frame]{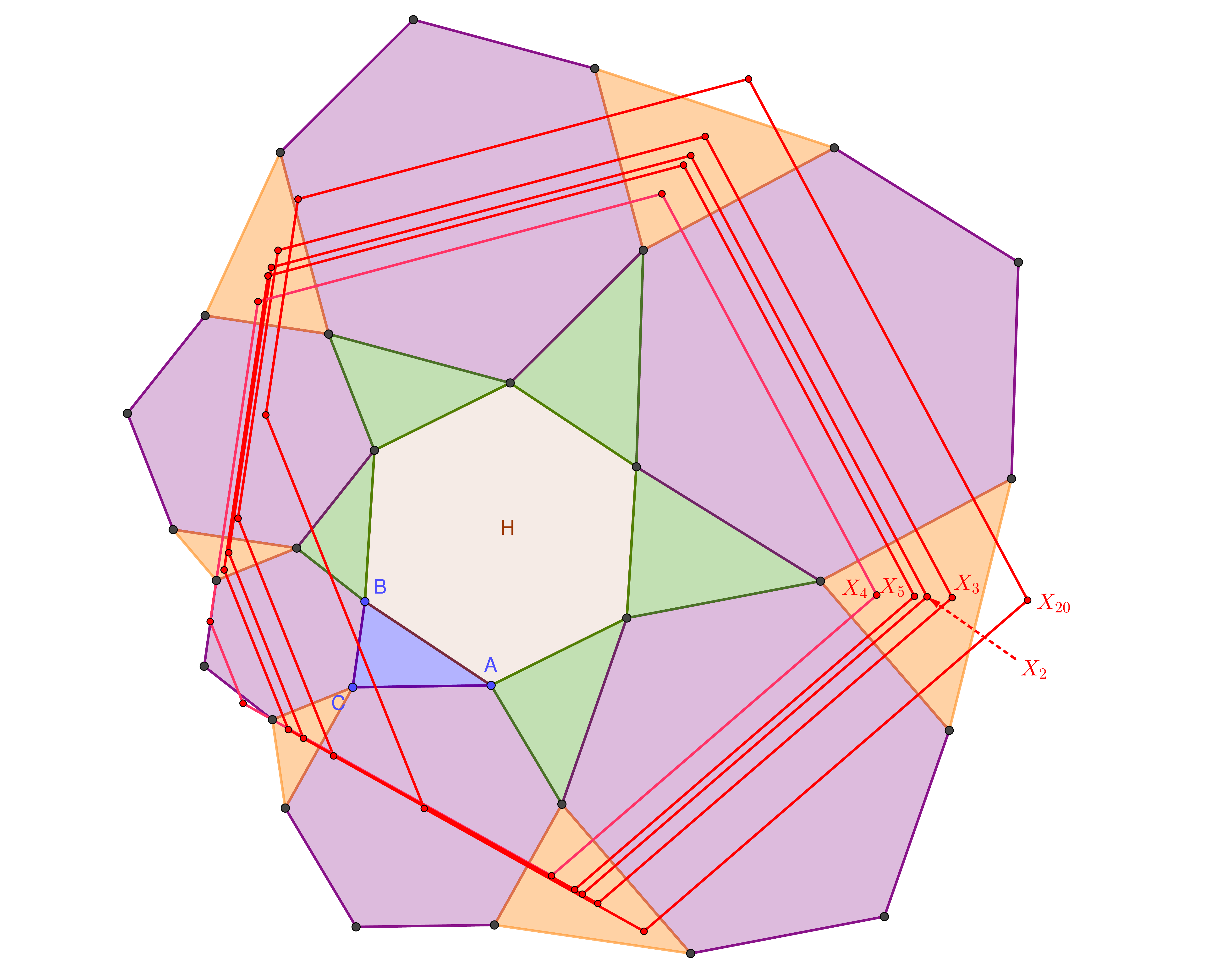}
    \caption{Shown are 5 hexagons $\H_i'$ (red) whose vertices are the $X_k$ of second-level satellites (yellow). If $X_k$ is on the Euler line, the $\H_i'$ have identical internal angles.}
    \label{fig:second-level}
\end{figure}

\subsection{Properties of the Second Fermat Point}

The second Fermat point $X_{14}$ of a triangle is the isogonal conjugate of the second isodynamic point $X_{16}$ \cite[Fermat Points]{mw}. Let $\H=ABCDEF$ be a regular hexagon with centroid $O$, and let $P$ be a point anywhere. Define six ``inner'' triangles $\T_1=ABP$, $\T_2=BCP$, ..., $\T_6=FAP$. Referring to \cref{fig:x14}(left):
The second Fermat point $X_{14}$ of a triangle is the isogonal conjugate of the second isodynamic point $X_{16}$ \cite[Fermat Points]{mw}. Let $\H=ABCDEF$ be a regular hexagon with centroid $O$, and let $P$ be a point anywhere.  $\T_1=ABP$, $\T_2=BCP$, ..., $\T_6=FAP$. Referring to \cref{fig:x14}(left):

\begin{proposition}
The $X_{14}$ of the $\T_i$ will lie on $PO$.
\end{proposition}


Let $\T_i'$ be the six triangles with (i) the base a side of $\H$, and (ii) the apex $P_i'$ the reflection of $P$ about said side. Assume in this case $P$ is interior to $\H$. Referring to \cref{fig:x14}(right):

\begin{proposition}
The $X_{14}$ of the $\T_i'$ lie on a rectangular hyperbola (green) concentric with $\H$.
\end{proposition}


\begin{figure}
\begin{subfigure}[m]{0.49\textwidth}
\includegraphics[trim=300 125 100 25,clip,width=\textwidth]{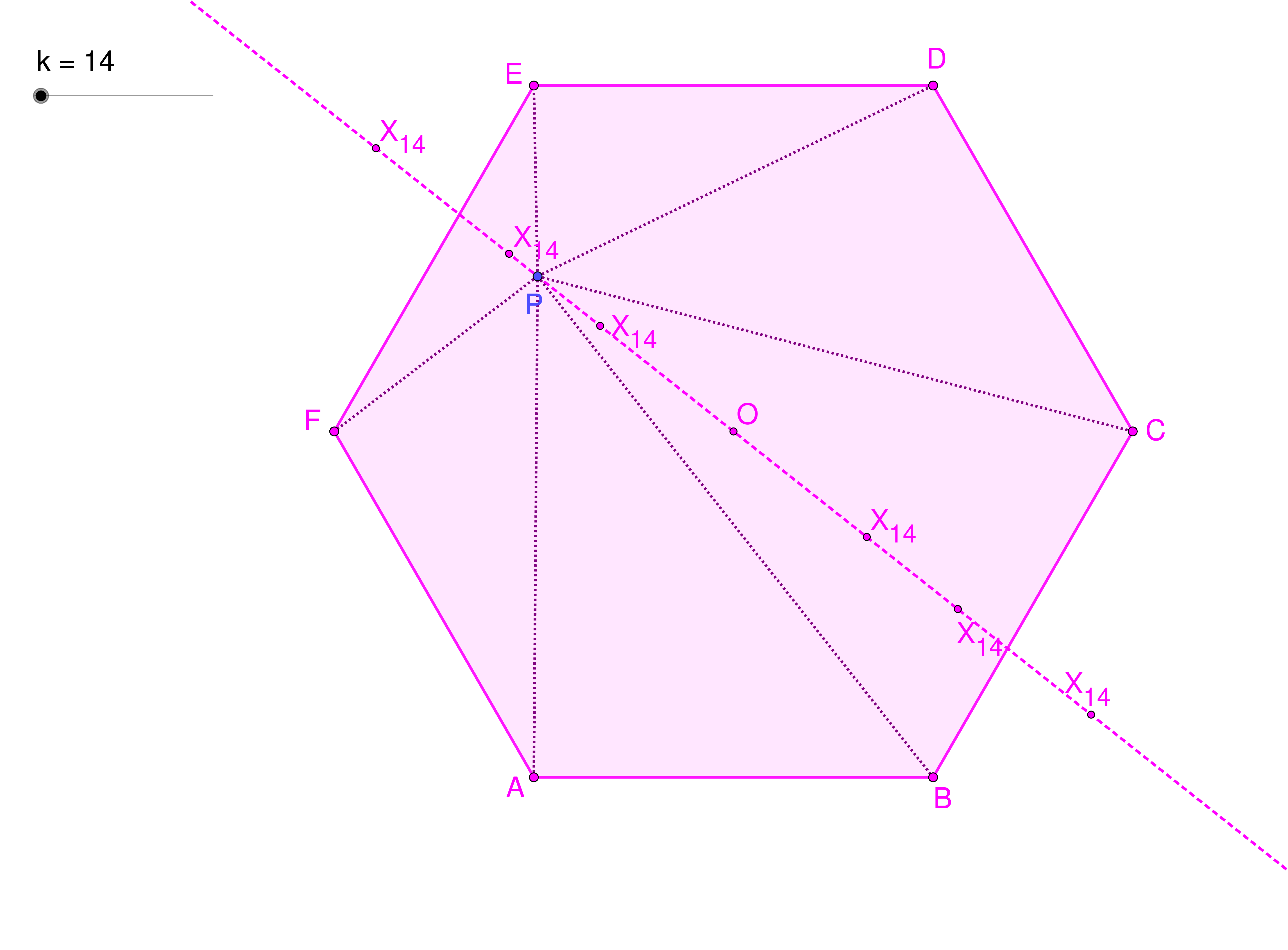}
\end{subfigure}
\begin{subfigure}[m]{0.49\textwidth}
\includegraphics[trim=200 0 0 0,clip,width=\textwidth]{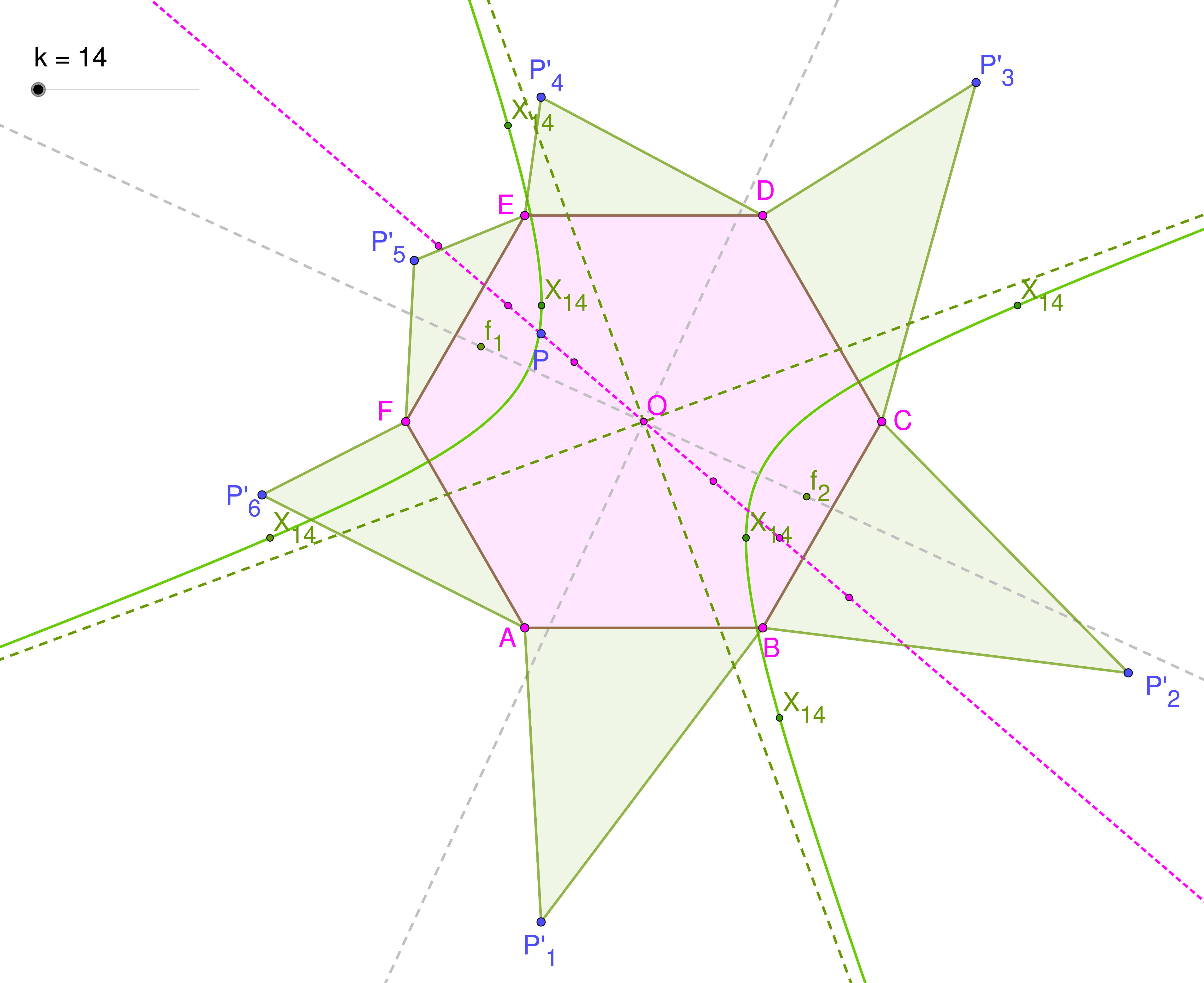}
\end{subfigure}
\caption{\textbf{Left:} Let $\H=ABCDEF$ be a regular hexagon, and $P$ a point. The second Fermat points $X_{14}$ of the ``inner'' triangles $\T_i\in\{ABP, BCP,\ldots FAP\}$ are collinear with $P$ and the centroid $O$ of $\H$. \textbf{Right:} Let $\T_i'$ be triangles with base a side of $\H$, and apex $P_i'$ the reflection of $P$ about said side. The $X_{14}$ of the $\T_i'$ lie on a rectangular hyperbola (green) concentric with $\H$. Also shown is the (dotted magenta) line of the $X_{14}$ of the inner triangles.}
\label{fig:x14}
\end{figure}

\section{A Web of Confocal Parabolas}
\label{sec:parabolas}
Referring to \cref{fig:parA}, let, let $\H_{i}$, $\H_{i+1}$, etc., be adjacent hexagons in the grid sharing antipodal vertices $U_i$, such that one of the $U_i$ is $A$. Let $\H'_i$, $\H'_{i+1}$, etc., be a second sequence of adjacent hexagons running along the same ``grain'' in the grid. The following was discovered by A. Akopyan \cite{akopyan2021-private}:

\begin{figure}
    \centering
    \includegraphics[width=\textwidth,frame]{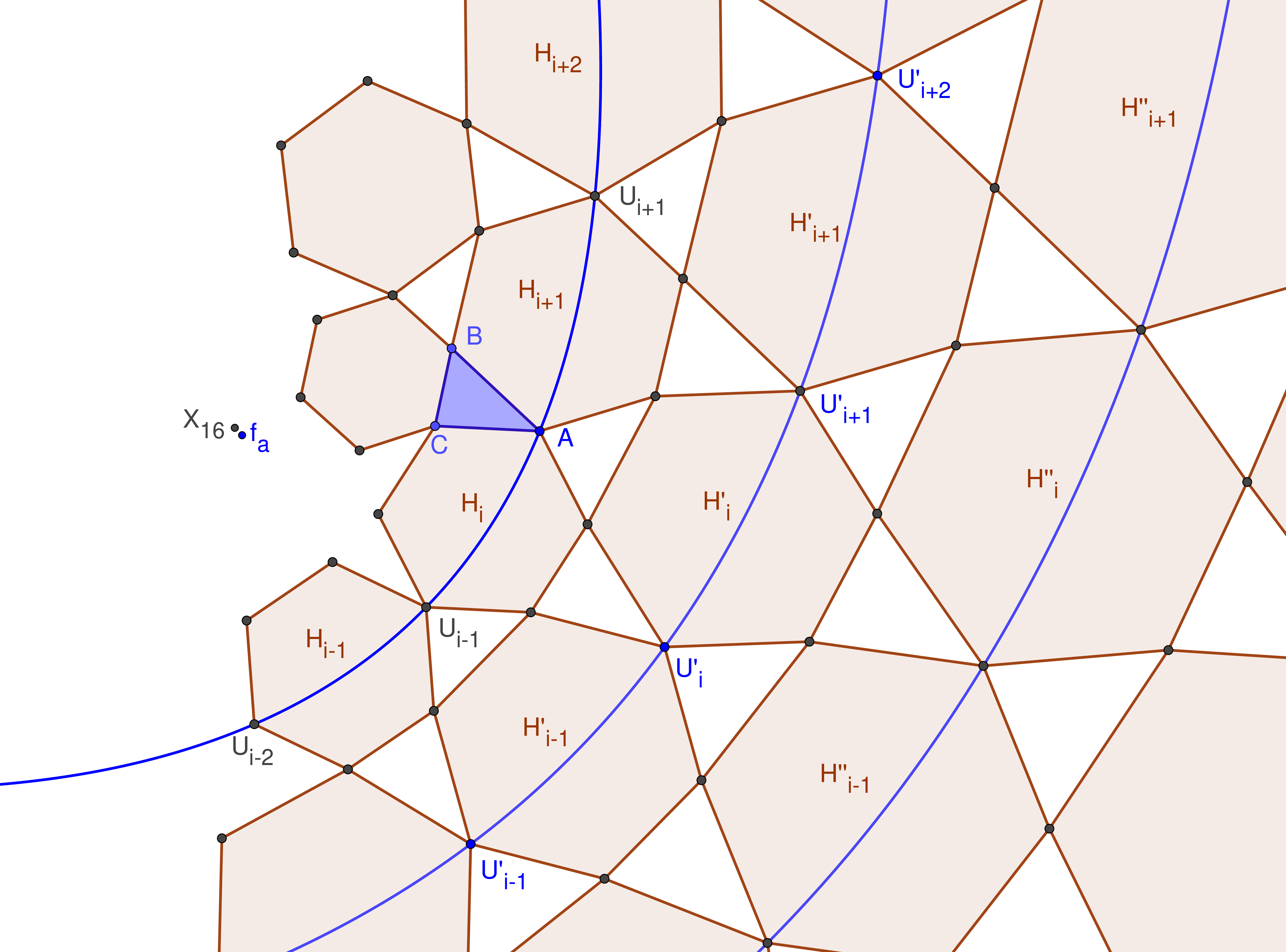}
    \caption{From hexagons $\H_{i-1}$, $\H_i$, etc., take antipodal vertices $U_{i-2}$, $U_{i-1}$, $A$, $U_{i+1}$. These lie on the $A$-parabola (blue), whose focus is $f_a$. Confocal parabolas pass through antipodal vertices along any other ``row'' of adjacent hexagons in the grid, e.g., $\H'_{i-1},\H'_i,\H_{i+},\ldots$.}
    \label{fig:parA}
\end{figure}

\begin{proposition}
The sequence $\ldots,U_{i-1},A,U_{i+1},\ldots$ lies on a parabola which we call the $A$-parabola. Furthermore, similarly-constructed sequences of vertices along hexagons in the same diagonal direction lie on parabolas confocal with the $A$-parabola at $f_a$.  
\end{proposition}

Referring to \cref{fig:par-major},
a total of three groups of confocal parabolas can be constructed, along three ``grains'' in the grid. Let $a,b,c$ denote the sides of $\T=ABC$, and $S$ is Conway's notation for twice the area of $\T$. Referring to \cref{fig:par-major-detail}:

\begin{theorem}
The foci of the three groups of confocal parabolas are vertices of an equilateral triangle with centroid at the $X_{16}$ common to all flanks and $\T$. The barycentric coordinates of $f_a$ are given by:
\[ f_a=\begin{bmatrix}
           \sqrt{3} (7 a^2 b^2 + 7 a^2 c^2 + 2 b^2 c^2 -4 a^4 - b^4- c^4) - 2 S (8 a^2 + b^2 + c^2) \\
           \sqrt{3} (2 a^2 b^2 - a^2 c^2 + 3 b^2 c^2 - 4 b^4 + c^4) + 2 S (2 a^2 - 2 b^2 - c^2) \\
           -\sqrt{3} (a^2 b^2 - 2 a^2 c^2 - 3 b^2 c^2 - b^4 + 4 c^4) + 2 S (2 a^2 - b^2 - 2 c^2)
         \end{bmatrix} \]
 Furthermore, side $s$ of the focal equilateral is given by:
\[ s^2 = \frac{3}{32}(a^2+b^2+c^2- 2 S \sqrt{3}) =\frac{3}{16}(\cot\omega-\sqrt{3})S \]
where $\omega$ is the Brocard angle of a triangle, given by $\cot(\omega)=(a^2+b^2+c^2)/(2S)$.
\label{thm:focal}
\end{theorem}

Let $a_i,b_i,c_i$ (resp $S_i$) be the sidelengths (resp. twice the area) of a given triangle $\T_i$ in the grid. Since any $\T_i$ can be used to start the grid:

\begin{corollary}
The quantity $a_i^2+b_i^2+c_i^2- 2 S_i\sqrt{3}$ is invariant over all $\T_i$.
\end{corollary}

As shown in \cref{fig:par-major-detail}:

\begin{proposition}
The axes of the three confocal groups concur at the common $X_{16}$ at $120^o$ angles.
\end{proposition}

\begin{figure}
    \centering
    \includegraphics[trim=100 200 1000 200,clip,width=\textwidth,frame]{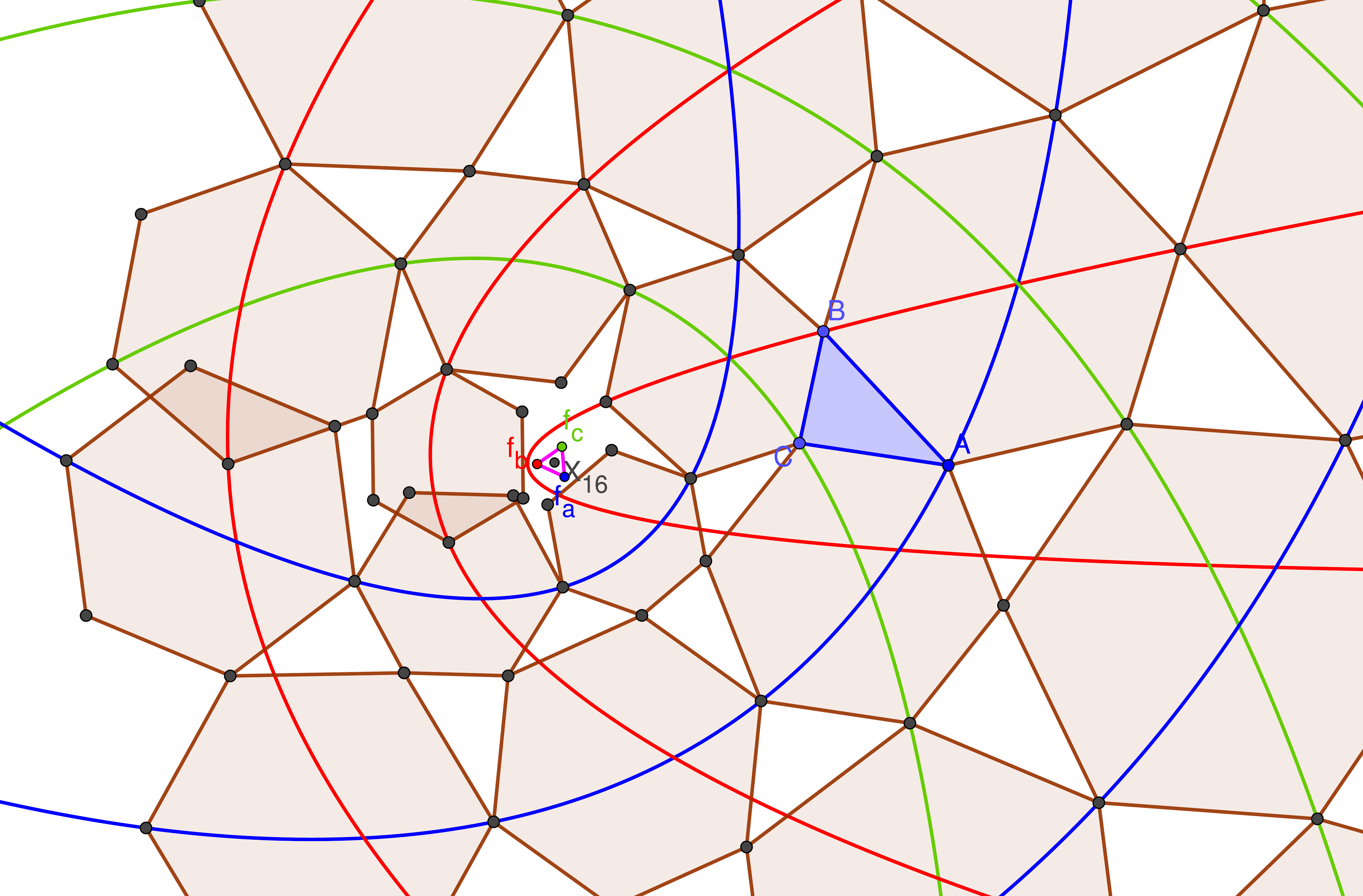}
    \caption{Three groups of confocal parabolas (red, green, blue) run along antipodal vertices of hexagons along three 3 major directions. The foci $f_a,f_b,f_c$ of each confocal group are the vertices of an equilateral triangle whose centroid is the $X_{16}$ common to all flank triangles in the grid.}
    \label{fig:par-major}
\end{figure}

\begin{figure}
    \centering
    \includegraphics[trim=400 0 300 0,clip,width=\textwidth,frame]{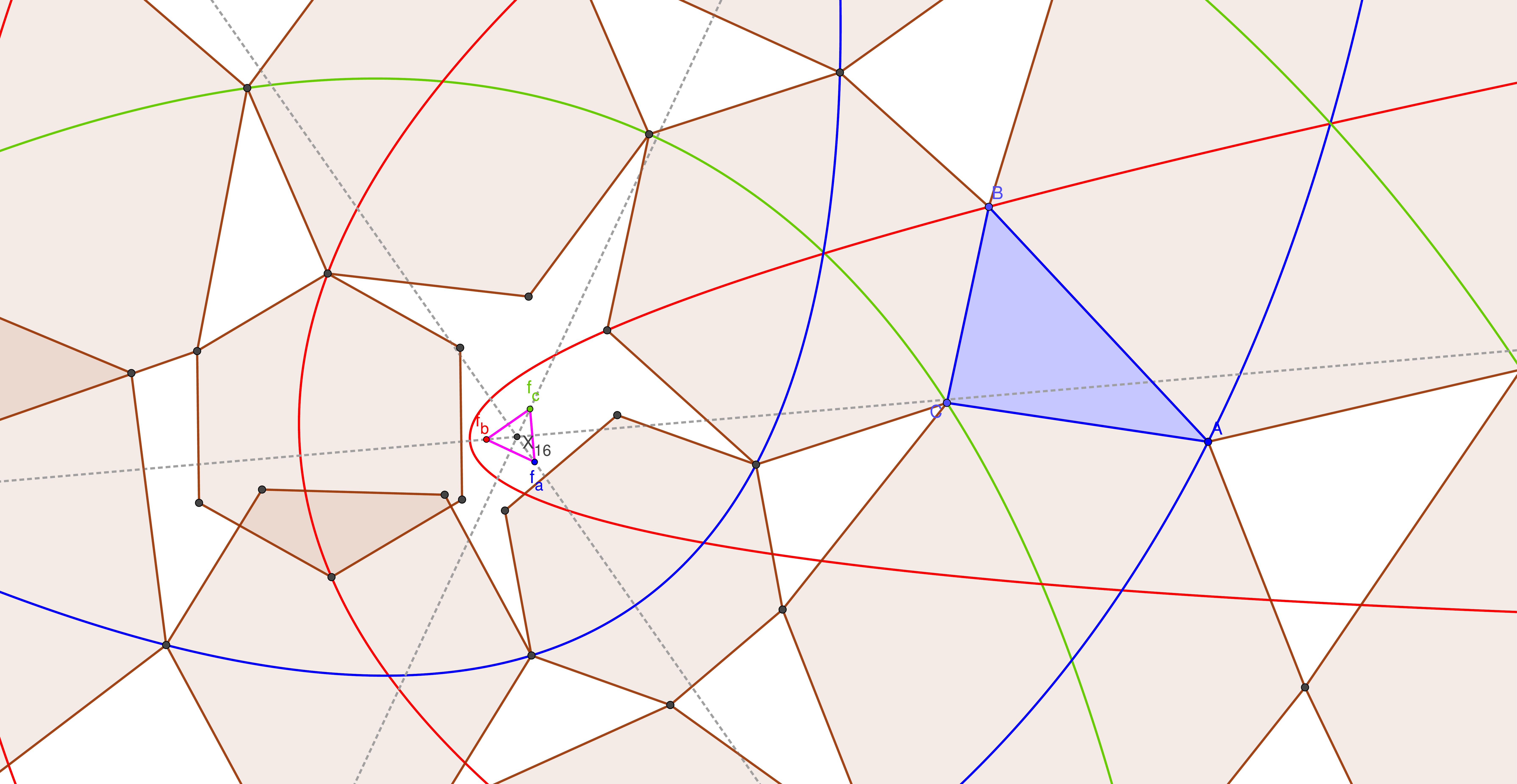}
    \caption{Zooming into the focal equilateral $f_a,f_b,f_c$. Also shown (dashed gray) are the common axes of the 3 confocal parabola groups, all of which pass throught the common $X_{16}$ at $120^o$ angles.}
    \label{fig:par-major-detail}
\end{figure}

\subsection*{A directrix equilateral}

The anticomplement\footnote{This is the double-length reflection about the barycenter $X_2$.} of the second Fermat point $X_{14}$ is labeled $X_{617}$ on \cite{etc}.

Referring to \cref{fig:dir-equi}:

\begin{proposition}
The triangle bounded by the directrices of the $A$-, $B$-, and $C$-parabolas is an equilateral whose centroid is $X_{617}$, and whose sidelengths $s'$ is given by:
\begin{align*}
    (s')^2 =& \left(\frac{S}{4}\right)\left(\frac{5\sqrt{3} + 11\cot{\omega} + 16(\sqrt{3} + \cot{\omega})}{2\cos{(2\omega)}-1}\right)
\end{align*}
\label{prop:dir-equi}
\end{proposition}
\noindent Note: an expression of the $A$-vertex of the above appears in \cref{app:eqns}. Note also that the sidelength is of the directrix equilateral is not conserved across all flank triangles in the grid since each of these will be associated with a different directrix equilateral.

Interestingly, the triangle formed by the vertices of said parabolas is not an equilateral.

\begin{figure}
    \centering
    \includegraphics[width=\textwidth]{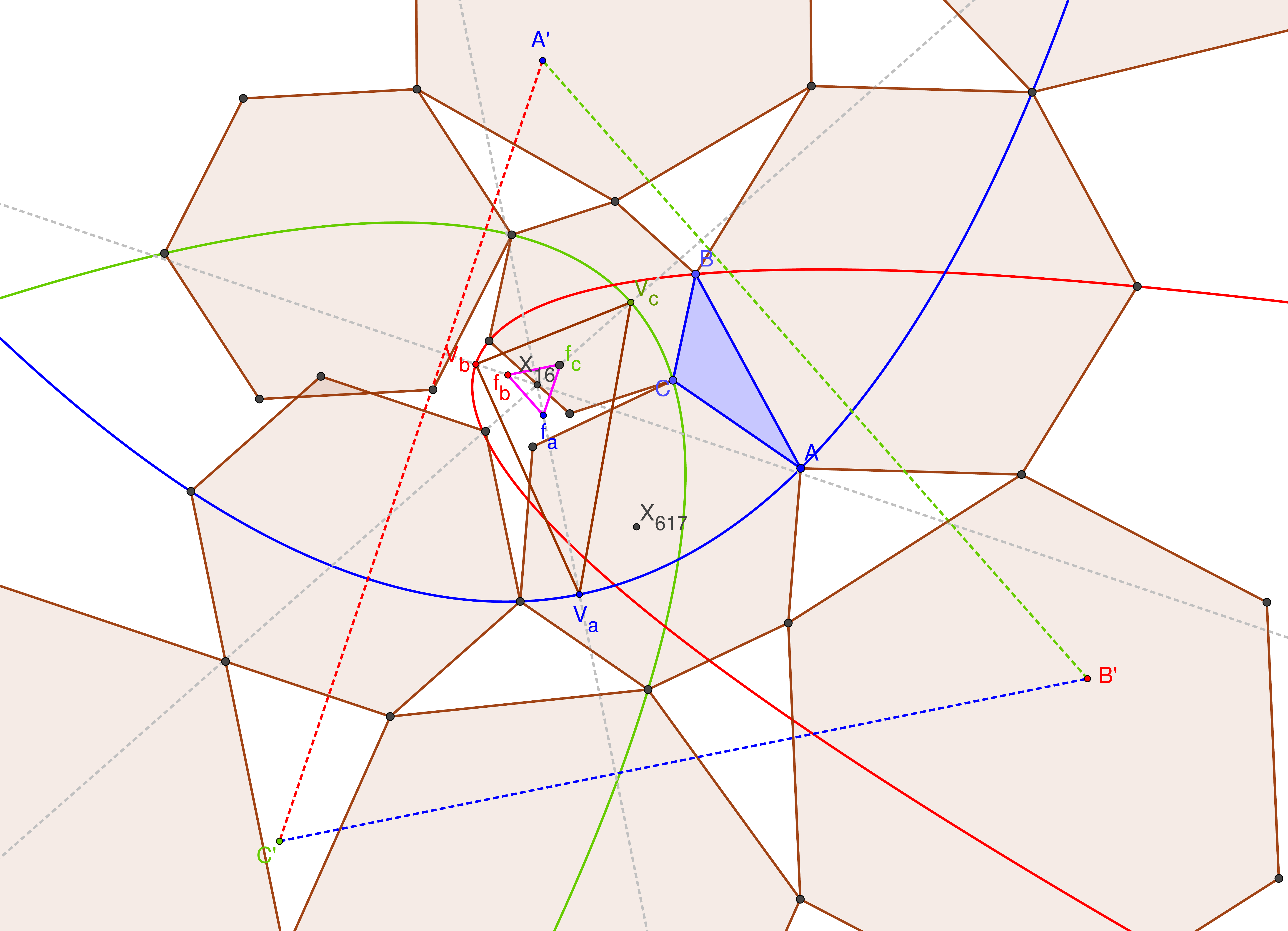}
    \caption{The triangle $A'B'C'$ bounded by the directrices (dashed blue, red, and green) of the $A$-, $B$-, and $C$-parabolas (blue, red, red, and green) is also an equilateral, whose centroid is $X_{617}$. Interestingly, the triangle (brown) connecting the vertices $V_a,V_b,V_c$ of said parabolas is in general a scalene.}
\label{fig:dir-equi}
\end{figure}

\subsection*{Skip-1 confocal parabolas}

Let $Q_1,A,Q_3,Q_5,\ldots$ (resp. $Q_2,B,Q_4,Q_6,\ldots$) be a sequence of odd (resp. even) side vertices of adjacent hexagons, as shown in  \cref{fig:par-alternated}. 

\begin{proposition}
The sequence of odd (resp. even) vertices lies on a parabola. The former (resp. latter) is confocal with the $A$-parabola (resp $B$-parabola). Furthermore, their axes are parallel to the axis of the $C$-parabola, and pass through $f_b$ and $f_c$, respectively.
\end{proposition}

\begin{figure}
    \centering
    \includegraphics[width=\textwidth,frame]{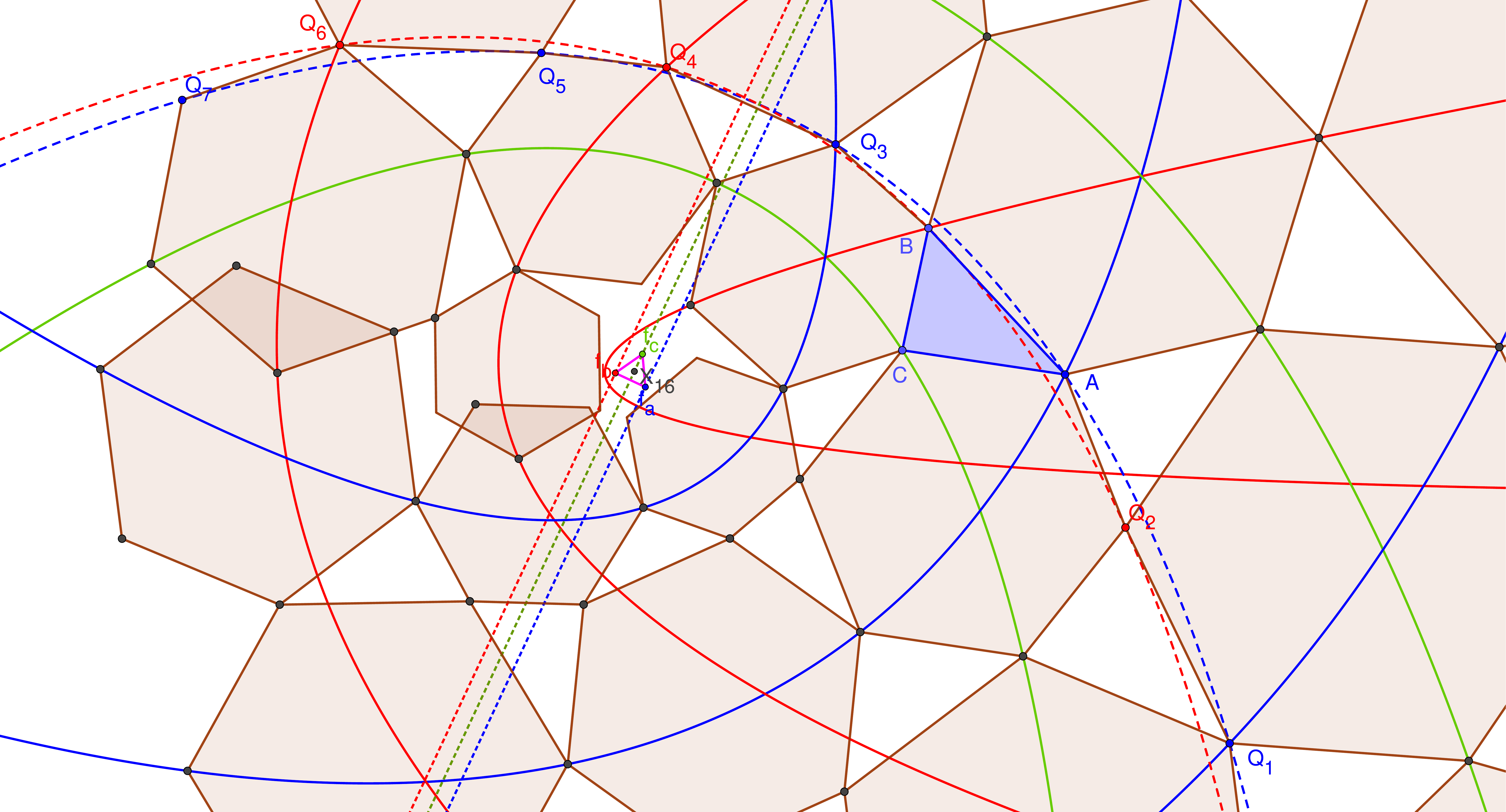}
    \caption{Two additional groups of confocal parabolas exist which pass through the odd (resp. even) side vertices along a given grain of the grid. A member of the first (resp. second) confocal group is shown in dashed blue (resp. red), passing through odd vertices $[\ldots,Q_1,A,Q_3,Q_5,Q_7,\ldots]$ (resp. even vertices) $[\ldots,Q_2,B,Q_4,Q_6,\ldots]$. The focus of the odd (resp. even) group is $f_a$ (resp. $f_b$). The major axes of the odd, even, and original $C$-parabola group are parallel (dashed red, blue, green).}
    \label{fig:par-alternated}
\end{figure}

The above statements are valid cyclically, i.e., there are families of confocal odd and even parabolas along each of the 3 major directions in the grid, e.g., corresponding to the diagonals of a hexagon erected upon a side of $\T$.

Computer-usable explicit equations for some of the objects in this section appear in \cref{app:eqns}.

\section{Controlled by Poncelet}
\label{sec:poncelet}
Recall Poncelet's theorem: if a polygon inscribed in one conic simultaneously circumscribes a second one, a 1d family of such polygons exists inscribed/circumscribed about the same conics \cite{dragovic2014}. In \cite{reznik2020-similarityII} two related families of Poncelet triangles are studied: (i) the homothetic family, and (ii) the Brocard porism, see \cref{fig:homot-broc}. Geometric details about this family are provided in \cref{tab:homot-broc}.

\begin{figure}
    \centering
    \includegraphics[width=\textwidth]{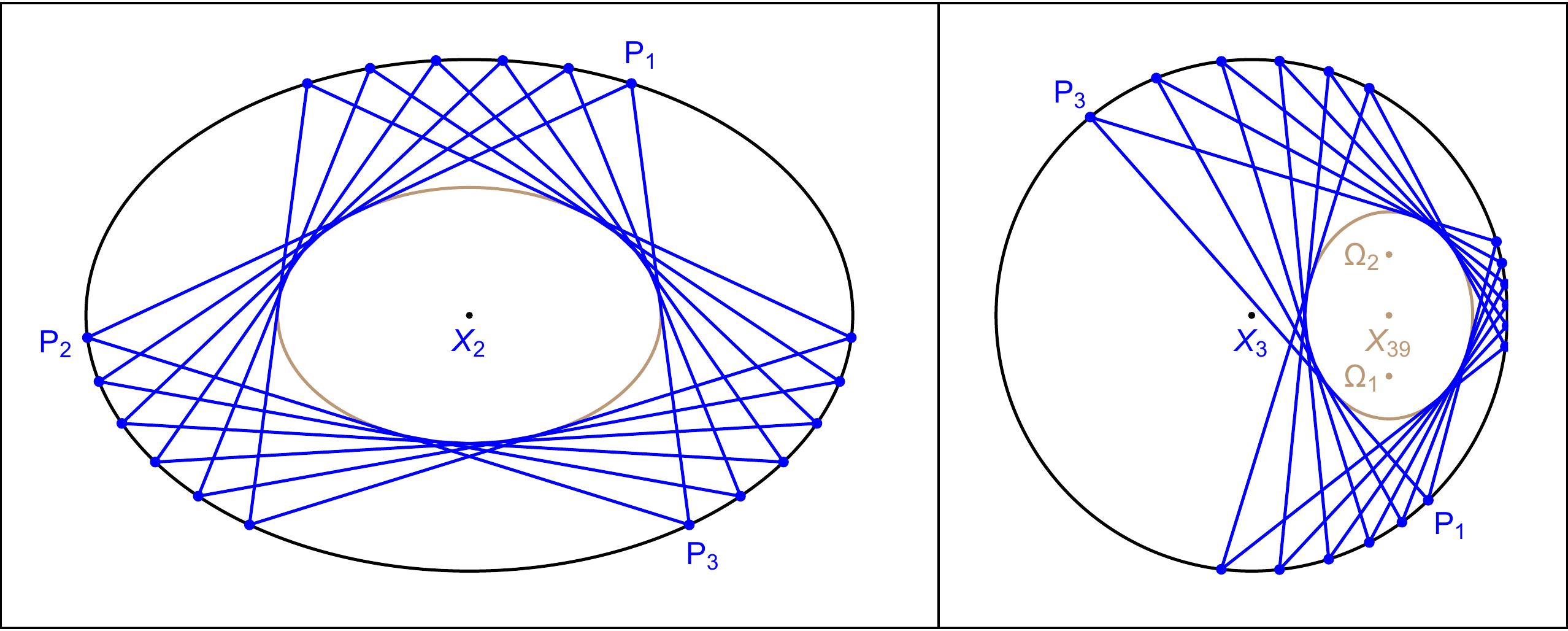}
    \caption{\textbf{Left:} The ``homothetic'' family is interscribed between two concentric, homothetic families centered on the fixed barycenter $X_2$. \textbf{Right:} The Brocard porism are circle-inscribed triangles circumscribing a fixed conic known as the Brocard inellipse, whose foci are the stationary Brocard points $\Omega_1,\Omega_2$ of the family.}
    \label{fig:homot-broc}
\end{figure}

\begin{table}[H]
\begin{tabular}{|r||l|l|l|l|}
\hline
family & Outer Conic & Inner Conic & Stationary & Conserves \\
\hline
Homothetic & Steiner Ellipse & Steiner Inellipse & $X_2$ & $\sum{s_i^2}$, $A$, $\omega$ \\
\hline
\makecell[lc]{Brocard\\Porism} & Circumcircle & Brocard Inellipse & \makecell[lc]{$X_3$, $X_6$, $X_{39}$,\\ $\Omega_1$, $\Omega_2$, $X_{15}$, $X_{16}$, $\ldots$} & $\sum{s_i^2}/A$, $\omega$\\
\hline
\end{tabular}
\caption{Geometric details about the homothetic and Brocard porism triangle families.}
\label{tab:homot-broc}
\end{table}

\subsection*{Homothetic phenomena}

Referring to \cref{fig:ponc-homot}, consider our basic construction such that the reference triangle $\T$ is one in the homothetic family.

Referring to the last column of \cref{tab:homot-broc}, notice that the homothetic family conserves both the sum of squared sidelengths and area. Another fact proved in \cite{reznik2020-similarityII} is that the locus of $X_k$, $k=13,14,15,16$ over this family are 4 distinct circles. 

Since the the side of said equilateral only depends on the sum of squared sidelengths and area \cref{thm:focal}:

\begin{corollary}
Over homothetic triangle family controlling our grid, the focal equilateral has invariant sidelength and circumradius. Furthermore, the locus of its centroid is a circle concentric with the homothetic pair of ellipses.
\end{corollary}

\begin{figure}
    \centering
    \includegraphics[width=.7\textwidth]{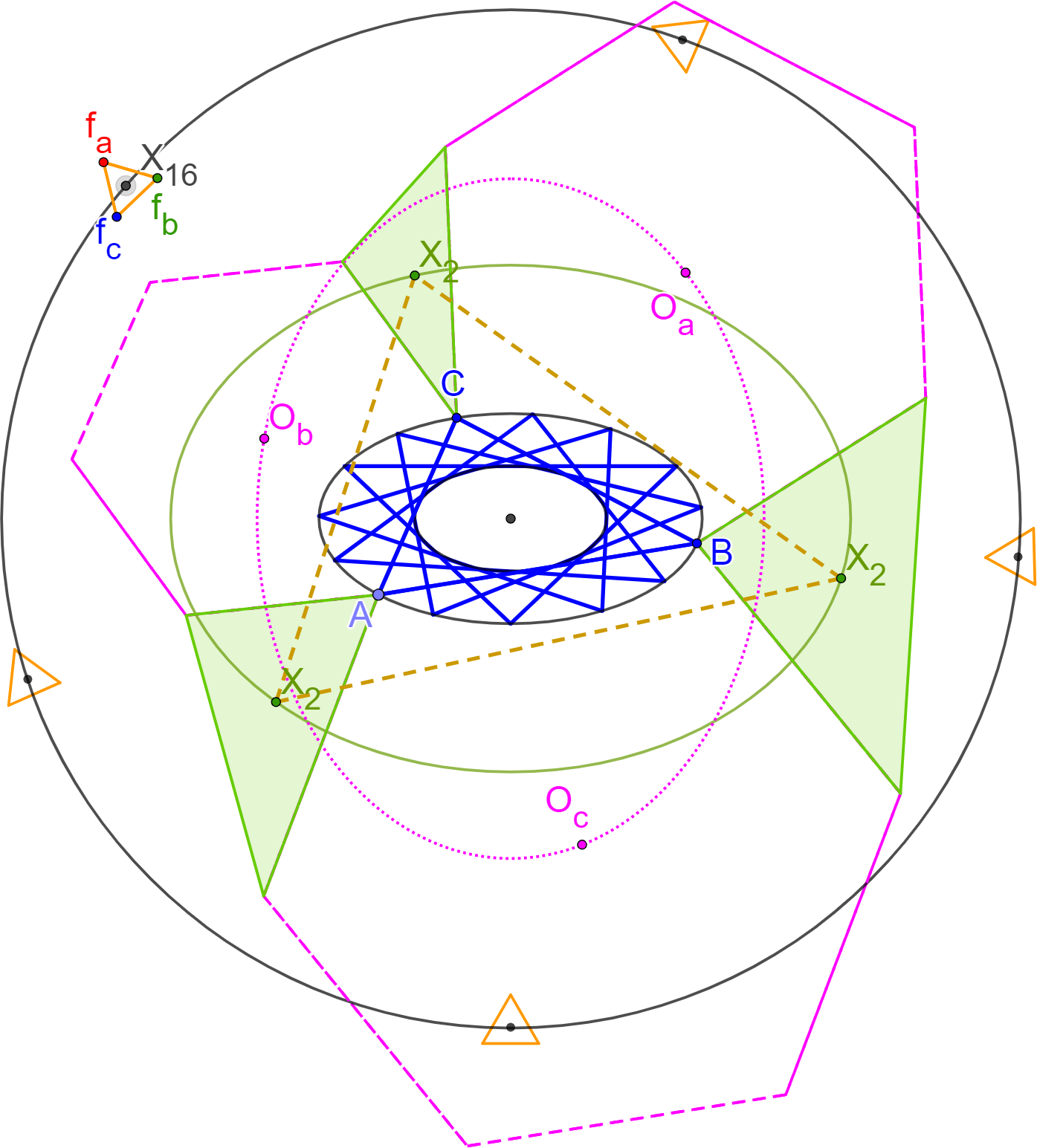}
    \caption{Phenomena manifested by objects in our basic construction over homothetic triangles (blue): (i) the focal equilateral (orange) has fixed sidelength and (ii) its centroid moves along a circle (black); (iii) the centroids $X_2$ of the 3 flank triangles move along a first ellipse (green) concentric with the homothetic pair; (iv) the centroids $O_a,O_b,O_c$ of the three regular hexagons (purple) move along a second ellipse (dotted purple) which is a $90^o$-rotated copy of (iii).}
    \label{fig:ponc-homot}
\end{figure}

Still referring to \cref{fig:ponc-homot}, experimentally, we observe:

\begin{observation}
Over the homothetic family, (i) the locus of the barycenters of the three flanks is an ellipse concentric and axis-alinged with the homothetic pair, though of distinct aspect ratio, and (ii) the locus of the centroids of the three regular hexagons erected on $\T$ is an ellipse which is a $90^o$-rotated copy of (i).
\end{observation}

\subsection*{Brocard porism phenomena}

Referring to \cref{fig:ponc-broc}, consider our basic construction such that the reference triangle $\T$ is one in the Brocard porism.

Let $\A=S/2$ denote the area of a reference triangle. Rewrite the expression for $s^2$ in  \cref{thm:focal} as $s^2=(3/8)\left(\cot{\omega}-\sqrt{3}\right)\A$. Referring to the last column of \cref{tab:homot-broc}, note the Brocard porism conserves $\omega$ (though not area), therefore:

\begin{corollary}
In a (dynamic) grid controlled by triangles $\T$ in the Brocard porism, the focal equilateral rotates about a fixed centroid $X_{16}$. Its area is variable and proportional to the area of $\T$.
\end{corollary}

\begin{figure}
    \centering
    \includegraphics[width=.8\textwidth]{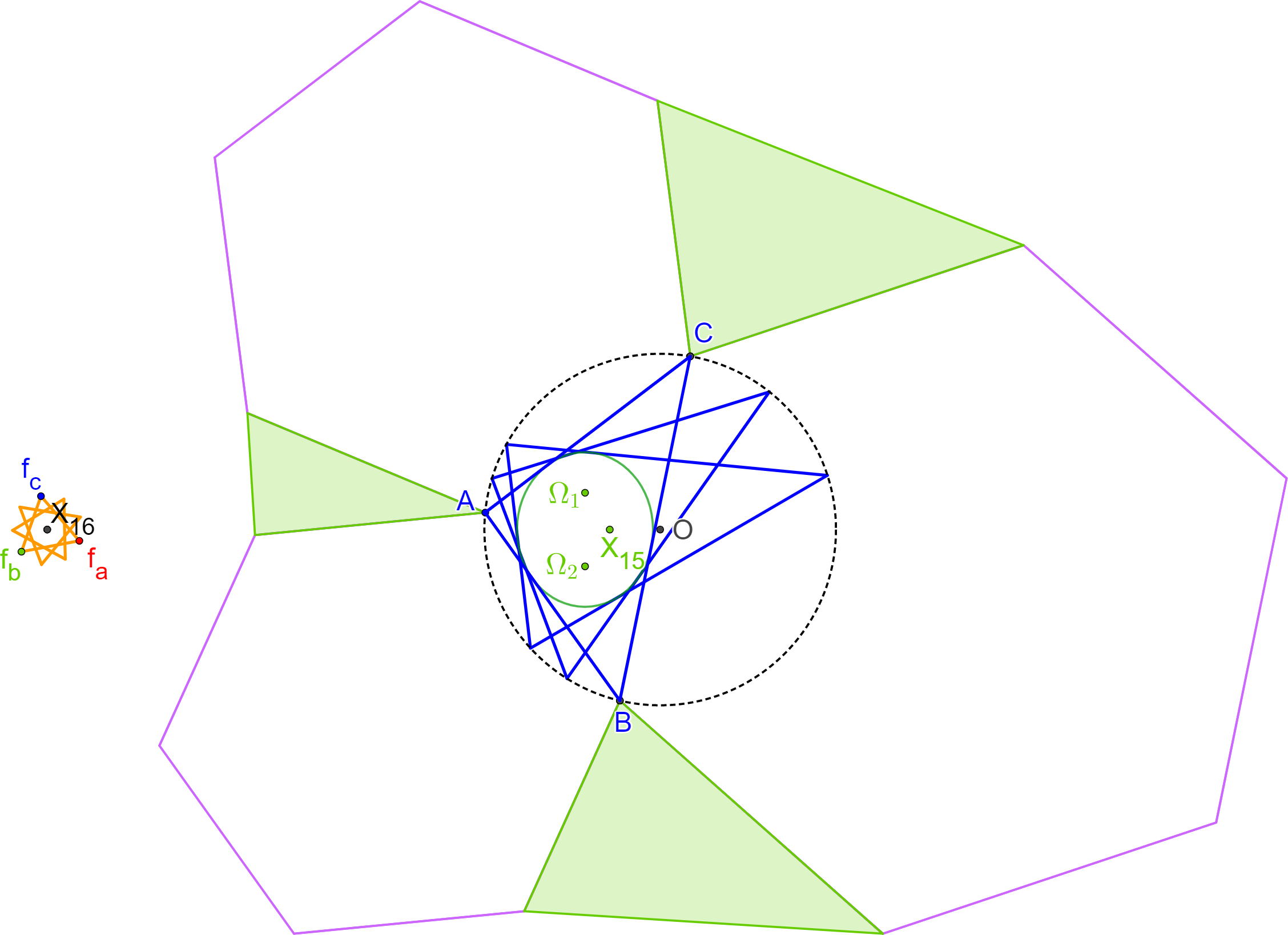}
    \caption{The isodynamic points $X_{15}$ and $X_{16}$ of Brocard porism triangles (blue) are stationary isodynamic. This entails that over the porism, flank triangle (green) $X_{16}$'s will also be stationary. Also shown are corresponding focal equilaterals, whose area is variable and proportional to the area of porism triangles.}
    \label{fig:ponc-broc}
\end{figure}

Since $X_6$ is stationary over the Brocard porism, recalling \cref{prop:x6}:

\begin{corollary}
Over the Brocard porism, the triangle whose vertices are the $X_{15}$ is perspective with $\T$ at a fixed point ($X_6$).
\end{corollary}

\section{Videos \& Questions}
\label{sec:videos}
Animations illustrating some constructions herein are listed on Table~\ref{tab:playlist}.

\begin{table}[H]
\small
\begin{tabular}{|c|l|l|}
\hline
id & Title & \textbf{youtu.be/<.>}\\
\hline
01 & {Basic construction of hexagonal flanks} &
\href{https://youtu.be/e3MkijszDEA}{\texttt{e3MkijszDEA}}\\
02 & \makecell[lc]{Circular Loci of $X_k$, $k=13,14,15,16$\\over the homothetic family} & \href{https://youtu.be/ZwTfwaJJitE}{\texttt{ZwTfwaJJitE}} \\
\hline
\end{tabular}
\caption{Videos of some constructions. The last column is clickable and provides the YouTube code.}
\label{tab:playlist}
\end{table}

\subsection*{Open Questions}

\begin{itemize}
    \item What dynamic properties underlie the fact that certain sequence of hexagonal vertices are spanned by parabolas? 
    \item Does the sequence of hexagons and flank triangles tend to regular shapes away from the foci of the parabolas?
    \item What are new or different properties of both the basic and grid constructions if hexagons are erected inwardly upon each side of $T=ABC$?
    \item Depending on the amount of self-intersection (\cref{fig:self-inter}) for one or more triangles in the grid, a certain condition  is crossed such that the second isodynamic point wanders away from its fixed common locations. What is that condition? Would using $X_{15}$ be correct?
    \item Is there an $N$ other than $3,4,6$ such that interesting properties of similar constructions can be found?
    \item What happens if self-intersected hexagons are erected, e.g., with vertices common with a simple regular hexagon?
\end{itemize}

\section*{Acknowledgements}
\noindent We thank Arseniy Akopyan for contributing crucial insights and discovering the web of parabolas in the grid. Darij Grinberg lend us an attentive ear and proved the stationarity of $X_{16}$. Clark Kimberling was kind enough to include our early observations on the Encyclopedia of Triangle Centers. 

\appendix
\section{Explicit Formulas}
\label{app:eqns}
To facilitate computational reproduction of our results, we provide code-friendly expressions for a few objects. In the expressions below $a,b,c$ refer to sidelengths and $x,y,z$ are barycentric coordinates.

\subsection{A-parabola} It is given implicitly by:

{\scriptsize
\begin{verbatim}
2*(a^8-4*a^6*b^2+6*a^4*b^4-4*a^2*b^6+b^8-4*a^6*c^2+13*a^4*b^2*c^2-14*a^2*b^4*c^2+5*b^6*c^2+
6*a^4*c^4-23*a^2*b^2*c^4+15*b^4*c^4-4*a^2*c^6+14*b^2*c^6+c^8)*x*y-3*c^2*(2*a^6-8*a^4*b^2+
10*a^2*b^4-4*b^6-5*a^4*c^2+18*a^2*b^2*c^2-11*b^4*c^2+4*a^2*c^4-8*b^2*c^4-c^6)*y^2+2*(a^8-
4*a^6*b^2+6*a^4*b^4-4*a^2*b^6+b^8-4*a^6*c^2+13*a^4*b^2*c^2-23*a^2*b^4*c^2+14*b^6*c^2+
6*a^4*c^4-14*a^2*b^2*c^4+15*b^4*c^4-4*a^2*c^6+5*b^2*c^6+c^8)*x*z+2*(4*a^8-13*a^6*b^2+
15*a^4*b^4-7*a^2*b^6+b^8-13*a^6*c^2+31*a^4*b^2*c^2-35*a^2*b^4*c^2+17*b^6*c^2+15*a^4*c^4-
35*a^2*b^2*c^4+36*b^4*c^4-7*a^2*c^6+17*b^2*c^6+c^8)*y*z-3*b^2*(2*a^6-5*a^4*b^2+4*a^2*b^4-
b^6-8*a^4*c^2+18*a^2*b^2*c^2-8*b^4*c^2+10*a^2*c^4-11*b^2*c^4-4*c^6)*z^2+2*sqrt(3)*S*(2*(a^6-
3*a^4*b^2+3*a^2*b^4-b^6-5*a^4*c^2+9*a^2*b^2*c^2-4*b^4*c^2+7*a^2*c^4-7*b^2*c^4-3*c^6)*x*y+
(a^2-b^2-c^2)*(2*a^4-4*a^2*b^2+2*b^4-10*a^2*c^2+8*b^2*c^2+5*c^4)*y^2+2*(a^6-5*a^4*b^2+
7*a^2*b^4-3*b^6-3*a^4*c^2+9*a^2*b^2*c^2-7*b^4*c^2+3*a^2*c^4-4*b^2*c^4-c^6)*x*z+2*(2*a^6-
7*a^4*b^2+8*a^2*b^4-3*b^6-7*a^4*c^2+17*a^2*b^2*c^2-12*b^4*c^2+8*a^2*c^4-12*b^2*c^4-3*c^6)*y*z+
(a^2-b^2-c^2)*(2*a^4-10*a^2*b^2+5*b^4-4*a^2*c^2+8*b^2*c^2+2*c^4)*z^2)=0
\end{verbatim}
}

The $B$- and $C$-parabolas can be obtained by cyclic permutations, i.e., $(a,b,c)\to(b,c,a)$, and $(a,b,c)\to(c,a,b)$, respectively.

\subsection{The two ``skip-1'' parabolas}

The skip-1 parabola through $B$ is given by:

{\scriptsize
\begin{verbatim}
(sqrt(3)*(2*a^6-5*a^4*b^2+4*a^2*b^4-b^6-7*a^2*b^2*c^2+3*b^4*c^2-3*a^2*c^4+3*b^2*c^4+c^6)+
6*(6*a^4-5*a^2*b^2+b^4-4*a^2*c^2+2*b^2*c^2+c^4)*S)*x^2+(sqrt(3)*(3*a^6-7*a^4*b^2+5*a^2*b^4-
b^6+6*a^4*c^2-10*a^2*b^2*c^2+2*b^4*c^2-3*a^2*c^4+5*b^2*c^4)+2*(13*a^4-8*a^2*b^2+b^4-
17*a^2*c^2+7*b^2*c^2+4*c^4)*S)*x*y+(sqrt(3)*(4*a^6-14*a^4*b^2+13*a^2*b^4-3*b^6-4*a^4*c^2-
13*a^2*b^2*c^2+4*b^4*c^2-a^2*c^4+4*b^2*c^4+c^6)+2*(22*a^4-14*a^2*b^2+b^4-8*a^2*c^2+
7*b^2*c^2+c^4)*S)*x*z+(sqrt(3)*(5*a^6-7*a^4*b^2+2*a^2*b^4+6*a^4*c^2+3*a^2*b^2*c^2-2*b^4*c^2-
6*a^2*c^4+b^2*c^4+c^6)-2*(11*a^4-10*a^2*b^2+2*b^4-a^2*c^2+2*b^2*c^2-c^4)*S)*y*z+
(sqrt(3)*(a^6-15*a^4*b^2+9*a^2*b^4-b^6-2*a^4*c^2-3*a^2*b^2*c^2-b^4*c^2+a^2*c^4+2*b^2*c^4)+
6*(3*a^4+2*a^2*b^2-b^4-a^2*c^2)*S)*z^2 = 0
\end{verbatim}
}

The skip-1 parabola through $C$ is obtained with a bicentric substitution, i.e., $(a,b,c,x,y,z)\to(a,c,b,x,z,y)$.

\subsection{Center (at infinity) of the A-parabola group} The $A$-parabola and $BC$ skip-1 pair of parabolas have parallel axes through $f_a,f_b,f_c$, therefore their axes will cross the line at infinity at the same point given by the following barycentrics;

{\scriptsize
\begin{verbatim}
x=8*a^6-8*a^4*b^2+a^2*b^4-b^6-8*a^4*c^2+6*a^2*b^2*c^2+b^4*c^2+a^2*c^4+b^2*c^4-c^6+
2*sqrt(3)*(b^2-c^2)^2*S

y=-4*a^6+a^4*b^2+a^2*b^4+2*b^6+7*a^4*c^2-3*a^2*b^2*c^2-5*b^4*c^2-2*a^2*c^4+4*b^2*c^4-
c^6-2*sqrt(3)*(a^2-c^2)*(b^2-c^2)*S

z=-4*a^6+7*a^4*b^2-2*a^2*b^4-b^6+a^4*c^2-3*a^2*b^2*c^2+4*b^4*c^2+a^2*c^4-5*b^2*c^4+
2*c^6+2*sqrt(3)*(a^2-b^2)*(b^2-c^2)*S
\end{verbatim}
}

\subsection{Directrix of the A-parabola} Is is the line given by:
{\scriptsize
\begin{verbatim}
(a^4-8*a^2*b^2+12*b^4-8*a^2*c^2+27*b^2*c^2+12*c^4-2*sqrt(3)*(2*a^2-5*b^2-5*c^2)*S)*x+
(6*a^4-23*a^2*b^2+22*b^4-30*a^2*c^2+58*b^2*c^2+39*c^4+2*sqrt(3)*(a^2-3*b^2-2*c^2)*S)*y+
(6*a^4-30*a^2*b^2+39*b^4-23*a^2*c^2+58*b^2*c^2+22*c^4+2*sqrt(3)*(a^2-2*b^2-3*c^2)*S)*z = 0
\end{verbatim}
}

The other two directrices can be obtained via cyclic substitution.

\subsection{Directrix equilateral}

The barycentrics $x,y,z$ of the $A$-vertex of the directrix equilateral (\cref{prop:dir-equi}) are given by:
{\scriptsize
\begin{verbatim}
x =-sqrt(3)*(12*a^6-13*a^4*b^2-a^2*b^4+2*b^6-13*a^4*c^2-8*a^2*b^2*c^2-2*b^4*c^2-a^2*c^4-
2*b^2*c^4+2*c^6)-6*(3*a^2*b^2+3*a^2*c^2+2*b^2*c^2)*S;

y = sqrt(3)*(5*a^6+2*a^4*b^2-10*a^2*b^4+3*b^6-7*a^4*c^2-10*a^2*b^2*c^2-a^2*c^4-6*b^2*c^4+
3*c^6)-6*(a^4-5*a^2*b^2+b^4-3*b^2*c^2-c^4)*S;

z = sqrt(3)*(5*a^6-7*a^4*b^2-a^2*b^4+3*b^6+2*a^4*c^2-10*a^2*b^2*c^2-6*b^4*c^2-10*a^2*c^4+
3*c^6)-6*(a^4-b^4-5*a^2*c^2-3*b^2*c^2+c^4)*S.
\end{verbatim}
}

\bibliographystyle{maa}
\bibliography{references,refs_pub,refs_sub}

\begin{thebibliography}{10}
\expandafter\ifx\csname urlstyle\endcsname\relax
 \providecommand{\url}[1]{doi:\discretionary{}{}{}#1}\else
 \providecommand{\url}{doi:\discretionary{}{}{}\begingroup
  \urlstyle{rm}\Url}\fi

\bibitem{akopyan2021-private}
Akopyan, A. (2021).
\newblock Private communication.

\bibitem{casey1888}
Casey, J. (1888).
\newblock \emph{A sequel to the first six books of the Elements of Euclid}.
\newblock Dublin: Hodges, {F}iggis \& {C}o., 5th ed.

\bibitem{cerin1998}
{\v{C}}erin, Z. (1998).
\newblock Regular hexagons associated to triangles with equal centroids.
\newblock \emph{Elem. Math.}, 53: 112--118.

\bibitem{cerin2002-flanks}
{\v{C}}erin, Z. (2002).
\newblock Loci related to variable flanks.
\newblock \emph{Forum Geometricorum}, 2: 105--113.

\bibitem{dosa2007-ext}
Dosa, T. (2007).
\newblock Some triangle centers associated with the excircles.
\newblock \emph{Forum Geometricorum}, 7: 151--158.

\bibitem{dragovic2014}
Dragović, V., Radnović, M. (2014).
\newblock Bicentennial of the great {P}oncelet theorem (1813--2013): current
  advances.
\newblock \emph{Bull. Amer. Math. Soc. (N.S.)}, 51(3): 373--445.

\bibitem{fukuta1996}
Fukuta, J. (1996).
\newblock Problem 10514.
\newblock \emph{Amer. Math. Monthly}, 103.
\newblock Solution: vol. 104 (1997), p.775.

\bibitem{fukuta1997}
Fukuta, J. (1996).
\newblock Problem 1493.
\newblock \emph{Math. Mag.}, 69.
\newblock Solution: 70 (1997), pp.70--73.

\bibitem{darij2021-private}
Grinberg, D. (2021).
\newblock Private communication.

\bibitem{hoehn2001-ext}
Hoehn, L. (2001).
\newblock Extriangles and excevians.
\newblock \emph{Math. Magazine}, 74: 384--388.

\bibitem{johnson17-schoutte}
Johnson, R. (1917).
\newblock Directed angles and inversion with a proof of {S}choute's theorem.
\newblock \emph{Am. Math. Monthly}, 24: 313--317.

\bibitem{johnson1960}
Johnson, R.~A. (1960).
\newblock \emph{Advanced Euclidean Geometry}.
\newblock New York, NY: Dover, 2nd ed.
\newblock Editor John W. Young.

\bibitem{etc}
Kimberling, C. (2019).
\newblock Encyclopedia of triangle centers.
\newblock \url{https://bit.ly/2GF4wmJ}.

\bibitem{lamoen2004-flank}
Lamoen, F.~v. (2001).
\newblock Friendship among triangle centers.
\newblock \emph{Forum Geometricorum}, 1: 1--6.

\bibitem{martini1996}
Martini, H. (1996).
\newblock On the theorem of {N}apoleon and related topics.
\newblock \emph{Mathematische Semesterberichte}, 43: 47--64.

\bibitem{fedor2022}
Morozov, E., Nilov, F. (2022).
\newblock Squares around affine regular and affine semi-regular polygons.
\newblock \emph{Kvant}.
\newblock To appear (in Russian).

\bibitem{reznik2020-similarityII}
Reznik, D., Garcia, R. (2021).
\newblock Related by similarity {II}: {P}oncelet 3-periodics in the homothetic
  pair and the {B}rocard porism.
\newblock \emph{Intl. J. Geom.}, 10(4): 18--31.

\bibitem{stachel2002}
Stachel, H. (2002).
\newblock Napoleon’s theorem and generalizations through linear maps.
\newblock \emph{Beiträge zur Algebra und Geometrie}, 43(2): 433--444.

\bibitem{mw}
Weisstein, E.~W. (2002).
\newblock \emph{{CRC concise encyclopedia of mathematics (2nd ed.)}}.
\newblock Boca Raton, FL: Chapman and Hall/CRC.

\end{thebibliography}

\end{document}